\documentclass[10pt]{amsart}
\usepackage[T1]{fontenc}
\usepackage{amsthm}
\usepackage{amssymb}
\usepackage{mathtools}
\usepackage{graphicx}
\usepackage{xcolor}
\usepackage{subcaption}
\usepackage[english]{babel}
\usepackage[autostyle]{csquotes}
\usepackage[hyperfootnotes=false]{hyperref}
\hypersetup{
    colorlinks=false,       
    linkbordercolor=blue,   
    citebordercolor=green,  
    pdftitle={Links of Mazur manifolds and exotica},
    pdfauthor={Sergey Nersisyan}
}

\newtheorem{prop}{Proposition}[section]
\newtheorem{cor}[prop]{Corollary}
\newtheorem{lem}[prop]{Lemma}
\newtheorem{thm}[prop]{Theorem}

\newtheorem{theorem}{Theorem}

\theoremstyle{definition}
\newtheorem{defn}[prop]{Definition}
\newtheorem{quest}[prop]{Question}

\theoremstyle{remark}
\newtheorem{remark}[prop]{Remark}

\numberwithin{equation}{section}
\usepackage{eucal}
\newcommand{\on}{\operatorname}
\newcommand{\A}{\mathcal A}
\newcommand{\B}{\mathcal B}
\newcommand{\C}{\mathcal C}
\newcommand{\M}{\mathcal M}
\newcommand{\BC}{{\mathbb C}}
\newcommand{\BP}{{\mathbb P}}
\newcommand{\BZ}{{\mathbb Z}}
\newcommand{\Id}{\operatorname{Id}}

\AtBeginDocument{\renewcommand{\coloneq}{\mathrel{\coloneqq}}}

\begin{document}
\title{Links of Mazur manifolds and exotica}
\author{Sergey Nersisyan}

\begin{abstract}
    In this paper, we explore links of Mazur manifolds in simple 4-manifolds. We construct non-split 2-component links in $S^4$. These are used to produce links in $\#^n \BC \BP^2$ that are split topologically but not smoothly. As a consequence, we obtain exotic pairs of simply connected, definite 4-manifolds with boundary, as well as exotic embeddings of various Mazur manifolds in $S^4$. 
\end{abstract}

\vspace*{-1cm}
\maketitle
\vspace*{-0.6cm}
\section{Introduction}\label{sec:introduction}

The study of exotica in dimension four has been one of the central topics in low-dimensional topology ever since the pioneering breakthroughs of Freedman \cite{Freedman} and Donaldson \cite{Donaldson}, providing positive and negative results, respectively. However, explicit constructions are often needed to apply these ideas. Perhaps the most useful tool for building smooth 4-manifolds is Kirby calculus, which originated in \cite{Kirby} and was fully developed in \cite{Gompf-Stipsicz}. Its first application to exotica appeared in \cite{Akbulut_cork}, where the different smooth structures are related by a twist of a contractible 4-manifold. It was later shown in \cite{Cork_thm} and \cite{Matveyev} that essentially all exotica can be localized in this way to a contractible piece, which came to be known as a cork. For this reason, it is important to understand smooth, contractible 4-manifolds and their embeddings in dimension four.

In this paper, we turn our attention to 2-component \enquote{links} of contractible 4-manifolds and study them through Kirby calculus.  
Recently, Golla and Marengon constructed non-split links of 2-spheres in $S^4$ that are nevertheless split by some homology 3-sphere \cite{Marcos}. Moving on to codimension-zero embeddings, we are interested in the analogous question for Mazur manifolds.

A smooth, compact, oriented 4-manifold $M$ is called \textit{Mazur}\footnote{While there is not necessarily a \textit{canonical} handle decomposition of $M$, we assume that a Mazur decomposition is given together with the Mazur manifold. Throughout this paper, $B^4$ is assumed to not be Mazur.} if it is contractible and admits a Kirby diagram consisting of one 0-handle, one 1-handle, and one 2-handle. See Figure~\ref{fig1} (left) for an example. We investigate the following question, keeping in mind both the smooth and topological categories.
\begin{quest}\label{Q}
Do there exist Mazur manifolds $M$ and $M'$ and an embedding of their disjoint union $$L : M \sqcup M' \hookrightarrow S^4,$$ such that $L$ is not split by $S^3$ but is split by some other homology 3-sphere?
\end{quest} 
Note that every such link is split by the boundaries $\partial M$ and $\partial M'$, which are homology spheres since they bound contractible manifolds. Thus, it would suffice to find a link $L$ not split by $S^3$.

A Mazur manifold is uniquely characterized by the attaching curve of its 2-handle, which is represented by a framed knot in $\partial(S^1 \times B^3) = S^1 \times S^2$ homotopic to the core $S^1 \times \{\mathrm{pt}\}$. For such a knot $K \subset S^1 \times S^2$,
we let $M(K)$ denote the corresponding Mazur manifold. We define the following relation between Mazur manifolds given by framed knots in $S^1 \times S^2$, before stating our first theorem.

\begin{defn}\label{ribbon}  Let $M = M(K)$ and $M' = M(K')$ be Mazur manifolds associated to framed knots $K, K' \subset S^1 \times S^2$. We say that $(M, M')$ is a \textit{ribbon pair} if there exists a ribbon concordance $$C \subset S^1 \times S^2 \times I$$ from $K$ to $K'$ preserving the framing. 
\end{defn}

\begin{theorem}\label{thm1}
Let $M = M(K)$ and $ M' = M(K')$ form a ribbon pair via a ribbon concordance $C \subset S^1 \times S^2 \times I$ from $K$ to $K'$.
There exists a smooth link $$L_C : M \sqcup -M' \hookrightarrow S^4$$ that is not split by a topological $S^3$. Moreover, the components $L_C|_M$ and $L_C|_{M'}$ are smoothly standard.
\end{theorem}

We define standard embeddings in $S^4$ in Definition~\ref{defn}. The complement $W_C = S^4 \backslash L_C$ is a ribbon homology cobordism from $\partial M$ to $\partial M'$, which allows us to apply the results of \cite{Metatheorem} (in particular, Theorem 1.19) to obstruct a \textit{smooth} splitting of $L_C$ by any homology sphere $Y$ that is \enquote{smaller} than $\partial M$. This leads us to the question of \textit{exotic} links, that is, links of Mazur manifolds that are split topologically but not smoothly.
However, the links $L_C$ are not even topologically split, and we are not able to construct links in $S^4$ which are exotic in this sense (see Question~\ref{Q1}). Instead, we achieve this in definite 4-manifolds.

Given a Mazur manifold $M = M(K)$ and an integer $n$, we denote by $M_n$ the Mazur manifold obtained from $M$ by increasing the framing of $K$ by $n$.  For $n < 0$, we write $\#^n \BC \BP^2$ for $\#^{-n} \overline{\BC \BP ^2}$.

\begin{theorem}\label{thm2} Let $M$, $M'$, and $C$ be as in Theorem~\ref{thm1}. For every $n \neq 0$, there exists a smooth link $$L_{C, n} : M_n \sqcup -M' \hookrightarrow \#^n \BC \BP^2$$ that is topologically but not smoothly split by $S^3$. Moreover, each component is smoothly standard, while the link is topologically standard.
\end{theorem}

We define standard knots and links in arbitrary 4-manifolds in Sections~\ref{2.3} and~\ref{2.4}. It is evident that the links $L_{C, n}$ possess some kind of exoticness, but they also lead to more conventional exotic structures. We discuss several such applications. 

\begin{figure}[htbp]
\centering
\includegraphics[width=0.25\textwidth]{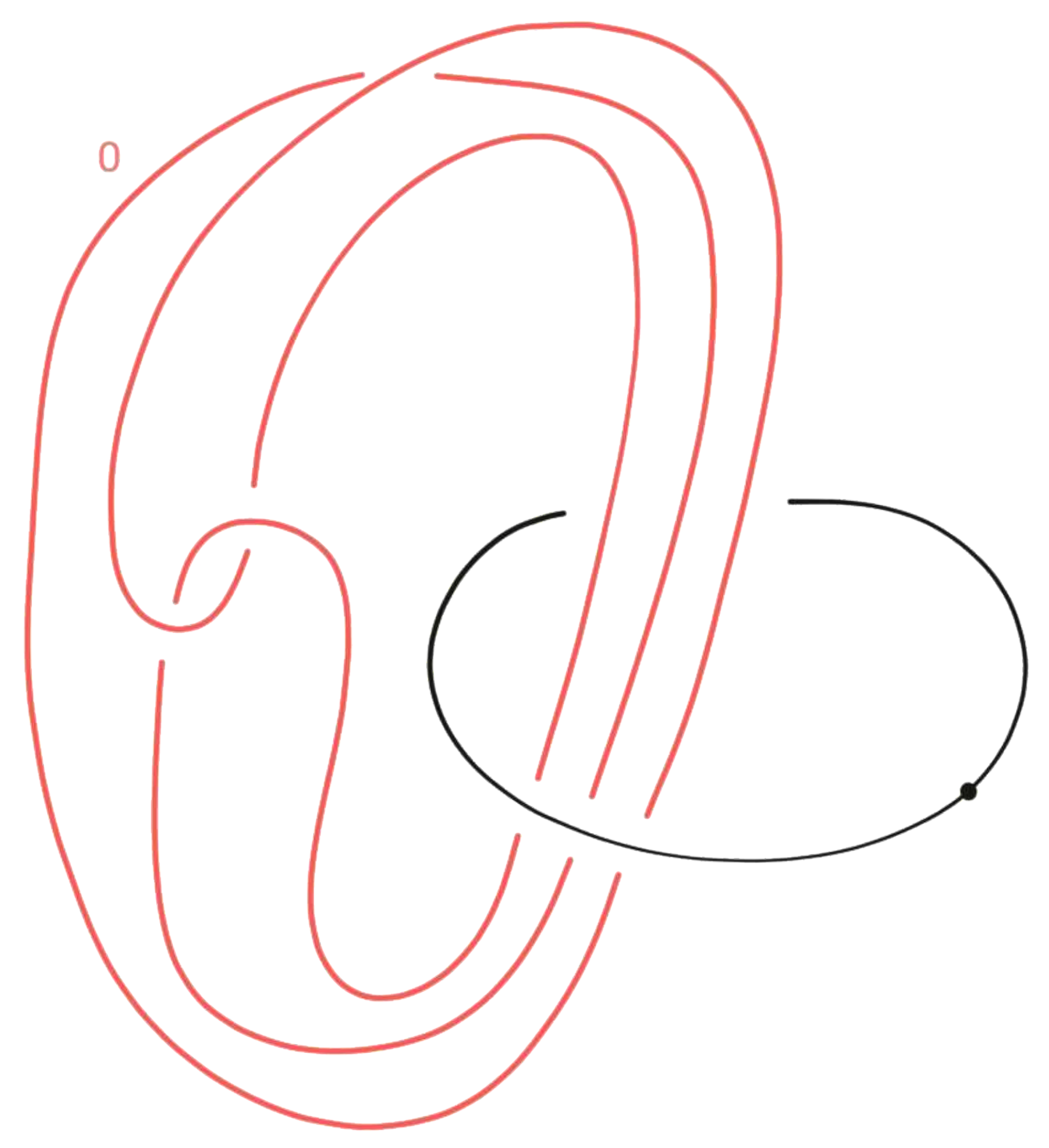} 
\hspace{2.5 cm}
\includegraphics[width=0.25\textwidth]{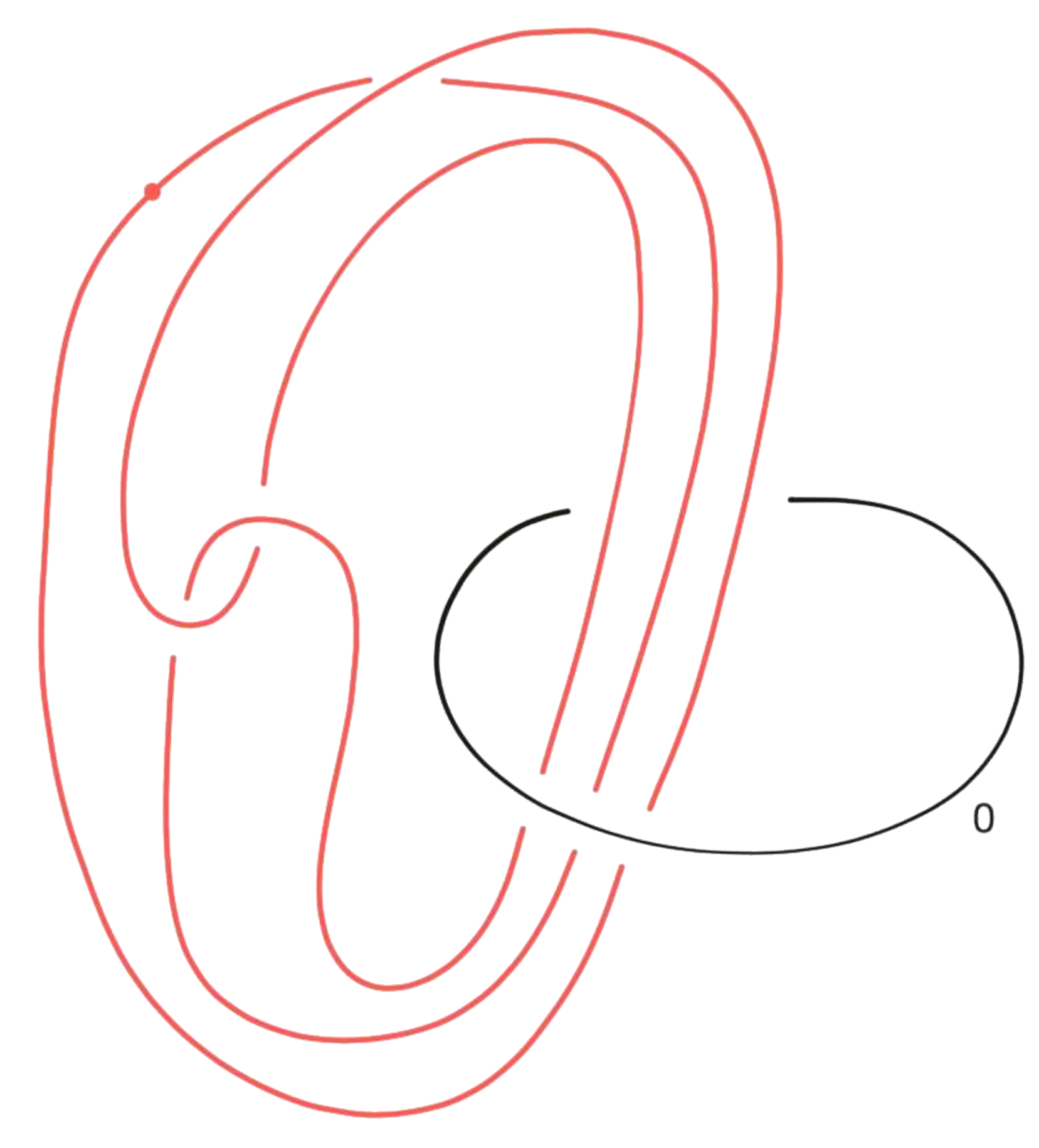}
\caption{Left: an example of a Mazur manifold $M$, this particular one being the Akbulut cork \cite{Akbulut_cork}. Right: the Mazur manifold obtained from $M$ by a zero-dot exchange when $M$ has an unknotted 2-handle.}
\label{fig1}
\end{figure}

\subsection{Link complements}\label{1.1}

We compare $L_{C, n}$ to the \textit{unlink} $$L_0 : M_n \sqcup -M' \hookrightarrow \#^n \BC \BP^2,$$ postponing a definition of an unlink of Mazur manifolds until Section~\ref{2.2}. It will follow from our proof of Theorem~\ref{thm2} that the two links are topologically isotopic. In particular, the complements $$W_{C, n} \coloneq \#^n \BC \BP^2 \backslash L_{C,n},\quad W_0 \coloneq \#^n \BC \BP^2 \backslash L_0$$
must be homeomorphic. On the other hand, they cannot be diffeomorphic since the latter contains a smooth separating $S^3$, while the former does not. We immediately obtain the following.
\begin{cor}\label{cor1}
    The simply connected, definite 4-manifolds with boundary $W_{C, n}$ and $W_0$ form an exotic pair.
\end{cor}

\begin{remark}
    These manifolds have two boundary components, namely $-\partial M_n$ and $\partial M'$. Deleting properly embedded arcs $\gamma_1 \subset W_{C, n}$ and $\gamma_2 \subset W_0$ connecting the two boundaries, we obtain 4-manifolds $W_{C, n}'$ and $W_0'$ with connected boundary $(-\partial M_n) \# \partial M'$. It can be checked that this pair is relatively exotic, meaning there is no diffeomorphism fixing the boundary. Then one can use the strategy of \cite{Akbulut-Ruberman} to obtain absolutely exotic pairs with connected boundary, which remain simply connected and definite.
\end{remark}

Furthermore, the existence of a topological isotopy $L_{C, n} \sim L_0$ also implies that the images of the following two embeddings lie in topological, but not smooth, 4-balls, leading to the corollary below.
$$M_n \hookrightarrow M_n \cup W_{C, n} = \#^n \BC \BP^2 \backslash L_{C, n}(M'), \quad -M' \hookrightarrow -M' \cup W_{C, n} = \#^n \BC \BP^2 \backslash L_{C, n}(M_n).$$

\begin{cor}\label{cor2}
For every Mazur manifold $M$ and a nonzero integer $n$, there exists a smooth 4-manifold with boundary $Z \overset{\mathrm{htpy}}{\simeq} \#^n \BC \BP^2 \backslash B^4$ and a smooth embedding $\varphi : M \hookrightarrow Z$ such that $\varphi(M) \subset Z$ can be isotoped inside a 4-ball topologically but not smoothly (compare with Proposition~\ref{propC}).
\end{cor}

This exoticness disappears when we glue in $M_n$ and $M'$: $$M_n \cup W_{C, n} \cup -M' = \#^n \BC \BP^2 = M_n \cup W_0 \cup -M'.$$ 
Moreover, since the restrictions $L_{C, n}|_{M_n}$ and $L_{C, n}|_{M'}$ are smoothly standard, we must have
$$W_{C, n} \cup -M' \cong W_0 \cup -M' = (-M_n) \#^n \BC \BP^2, \quad M_n \cup W_{C, n} \cong M_n \cup W_0 = M' \#^n \BC \BP^2.$$
We conclude that the exoticness of $W_{C, n}$ and $W_0$ disappears even after we glue in only one of $M_n, M'$.\\

This means that if we are interested in exotica that survives a single gluing, we must first modify the links $L_C$ and $L_{C, n}$. We are able to do this in the case when $M$ belongs to a certain subfamily of Mazur manifolds, which is known to exhibit some exotic behavior, as we discuss in the next section.
\subsection{Zero-dot exchange and Stein structures}\label{1.2}
Let $K_0 \subset S^1 \times S^2$ be a 0-framed knot with winding number $1$ that appears unknotted in the surgery diagram of $S^1 \times S^2$. Setting $M_0 = M(K_0)$, the 2-handle of $M_0$ is 0-framed and unknotted, which allows us to perform a zero-dot exchange (see \cite{Gompf-Stipsicz}) to obtain the dual Mazur manifold $M^{\bullet}$, with a boundary diffeomorphism $\theta : \partial M_0 \rightarrow \partial M^{\bullet}$ (see Figure~\ref{fig1}). Moreover, there are embeddings of $M_0$ and $M^{\bullet}$ into $S^4$ coming from the following decomposition:
$$S^4 = M_0 \cup_{\theta} -M^{\bullet}.$$

For such a Mazur manifold $M_0 = M(K_0)$, let $\on{tb}(M_0)$ denote the maximal Thurston--Bennequin invariant (see \cite{Etnyre}) of a Legendrian realization of $K_0$ in the standard contact structure of $S^1 \times S^2$. It follows from the Eliashberg criterion \cite[Theorem~11.2.2]{Gompf-Stipsicz} that $M_0$ admits a Stein structure if $\on{tb}(M_0) > 0$. In that case, it can be shown (see \cite{Akbulut-Yasui}) that the boundary map $\theta : \partial M_0 \rightarrow \partial M^{\bullet}$ does not extend to a diffeomorphism $\Theta : M_0 \xrightarrow{\cong} M^{\bullet}$. On the other hand, $\theta$ must always extend to a homeomorphism $M_0 \approx M^{\bullet}$ by \cite{Freedman}. It follows that when $\on{tb}(M_0) > 0$, the embeddings coming from the zero-dot exchange are exotic (see Proposition~\ref{propEx}).\\

The motivation for considering such manifolds $M_0$ is that we can apply the replacement $M_0 \rightsquigarrow_{\theta} M^{\bullet}$ (delete $M_0$ and glue in $M^{\bullet}$ using $\theta$) to the links $L_C$ in the following way.

Let $C_0$ be a ribbon concordance from $K_0$ to some knot $K'$, set $M' = M(K')$ as before, and let
$$L_{C_0} : M_0 \sqcup -M' \hookrightarrow S^4$$
be the link provided by Theorem~\ref{thm1}. Replacing $M_0$ by $M^{\bullet}$ via the map $\theta$, we obtain a link
$$L_{C_0}^{\bullet} : M^{\bullet} \sqcup -M' \hookrightarrow S^4,$$
where the ambient space remains $S^4$ because we replace the decomposition $M_0 \cup_{\Id} -M_0$ by $M^{\bullet} \cup_{\theta} -M_0$.

Applying the construction of Theorem~\ref{thm2} to $L_{C_0}^{\bullet}$, we obtain links $$L_{C_n}^{\bullet} : M^{\bullet}_n \sqcup -M' \hookrightarrow \#^n \BC \BP^2,$$ where we write $L_{C_n}^{\bullet}$ instead of $L_{C_0, n}^{\bullet}$ because these links are not obtained from $L_{C_0, n}$.

Just like in the case of $L_{C, n}$, it will follow that $L_{C_n}^{\bullet}$ is split by a topological $S^3$, but not by a smooth one. However, while the restriction $L_{C, n}|_{M_n}$ is smoothly standard because $L_C|_M$ is standard, the restriction $L_{C_n}^{\bullet}|_{M^{\bullet}_n}$ can be more interesting when $L_{C_0}^{\bullet}|_{M^{\bullet}}$ is smoothly nonstandard, i.e., $\theta : \partial M_0 \rightarrow \partial M^{\bullet}$ does not extend. While this is a necessary condition, it may not be sufficient because the exoticness of $\theta$, and hence of $L_{C_0}^{\bullet}$, may disappear when we \enquote{blow up} to obtain the links $L_{C_n}^{\bullet}$.
One way to ensure that the exotica survives the blow-up is to consider Stein structures.

\begin{theorem}\label{thm3}
Let $M_0$, $M'$, and $C_0$ be as above and assume that $\on{tb}(-M_0) > 0$. Then, for $n < 0$, the complement $\#^n \BC \BP^2 \backslash L_{C_n}^{\bullet}(M_n^{\bullet})$ admits a Stein structure. The complement of the standard embedding $M_n^{\bullet} \hookrightarrow \#^n \BC \BP^2$ is $-M_n^{\bullet} \#^n \BC \BP^2$, which is not Stein. Thus, the two complements form an exotic pair of 4-manifolds with boundary. Moreover, this pair is related by a replacement $M' \rightsquigarrow S^4 \backslash L_{C_0}^{\bullet}(-M')$.
\end{theorem}

Although we obtain exotic pairs of simply connected, definite 4-manifolds with connected boundary, it is likely that many of these were known to the experts. Indeed, Akbulut mentions these pairs for the case when $M_0$ is the cork from \cite{Akbulut_cork} in Remark 1 of \cite{Akbulut_definite}. But while the general case is no more complicated, it does not appear in the literature, to the best of our knowledge. Moreover, these pairs arise naturally from our links $L_{C_n}^{\bullet}$, and can be closed off to obtain $\#^n \BC \BP^2$. Finally, the exotic pairs only depend on $M_0$ and $n$, but they are related by a replacement of $M'$, which can be taken to be any Mazur manifold that is ribbon with respect to $M_0$. 

Another example of exotic pairs of definite 4-manifolds with boundary appears in recent work of Cavallo \cite[Theorem~1.4]{Cavallo}. Inspired by their Corollary 1.5, we obtain the following:

\begin{cor}\label{cor3}
For every $n \neq 0$ and $M_0$ satisfying $\on{tb}(-M_0) > 0$, there exist two smooth embeddings $f, g : M_n^{\bullet} \hookrightarrow \#^n \BC \BP^2$ such that no diffeomorphism of $\#^n \BC \BP^2$ maps $f(\partial M_n^{\bullet})$ to $g(\partial M_n^{\bullet})$.
\end{cor}

\subsection{Mazur knots}\label{1.3}

Earlier, we observed that when the boundary map $\theta : \partial M_0 \rightarrow \partial M^{\bullet}$ does not extend to a diffeomorphism $\Theta : M_0 \xrightarrow{\cong} M^{\bullet}$, we obtain exotic embeddings of $M_0$ and $M^{\bullet}$ in $S^4$. We are also interested in general Mazur manifolds that admit exotic embeddings in $S^4$.

Assume $M$ is Mazur and $\varphi : M \hookrightarrow S^4$ is a smooth embedding that is topologically but not smoothly isotopic to the \textit{standard} embedding $\varphi_M : M \hookrightarrow S^4$. According to Proposition~\ref{propEx} below, this is equivalent to requiring that the complement $N \coloneq -S^4 \backslash \varphi(M)$ is homeomorphic to $M$, but the boundary map (which we continue to denote $\theta$) identifying $\partial M$ and $\partial N$ does not extend to a diffeomorphism $\Theta : M \xrightarrow{\cong} N$. In addition, we assume that $N$ is an AC manifold as defined in \cite{Melvin-Schwartz}.

We prove the following result, which can be interpreted as saying that exotica travels across a ribbon concordance.

\begin{theorem}\label{thm4}
For every Mazur manifold $M'$ that is ribbon with respect to $M$, there exists an exotic embedding $\varphi' : M' \hookrightarrow S^4$. In other words, setting $N' \coloneq -S^4 \backslash \varphi'(M')$ and letting $\theta' : \partial M' \rightarrow \partial N'$ be the induced boundary diffeomorphism, $\theta'$ does not extend to a diffeomorphism $\Theta' : M' \xrightarrow{\cong} N'$.
\end{theorem}

Taking a link $L_C : M \sqcup  -M' \hookrightarrow S^4$ from Theorem~\ref{thm1}, we produce a link $L_{C, \theta} : N \sqcup -M' \hookrightarrow S^4$ by performing a replacement $M \rightsquigarrow_{\theta} N$. The assumption that $\theta$ does not extend readily implies that the embedding $L_{C, \theta}|_N$ is exotic. We prove that the embedding $\varphi' \coloneq L_{C, \theta}|_{M'}$ is also exotic by showing that the replacement $M' \rightsquigarrow_{\theta'} N'$ changes the induced decomposition of $S^4$ as follows:
$$S^4 = N \cup_{\theta} W_C \cup -M' = N \cup_{\theta} -M \rightsquigarrow S^4 = N \cup_{\theta} W_C \cup_{\Id} -W_C \cup_{\theta} -N = N \cup_{\Id} -N.$$

Here we do not show that $W_C \cup -W_C = \partial N \times I$, but only that $N \cup W_C \cup -W_C = N$, making explicit use of the AC structure of $N$. One could try to employ the same ideas to obtain the backwards direction of the theorem: use an exotic embedding $\varphi'$ of $M'$ to get an exotic embedding $\varphi$ of $M$. However, we are unable to show $N' \cup -W_C \cup W_C \cong N'$ because the two copies of $W_C$ are glued with the \enquote{wrong} end.

\begin{remark}\label{rem1}
When the manifolds $M$ and $N$ are not diffeomorphic at all, $\theta$ certainly does not extend to a diffeomorphism $\Theta$. For example, one can choose $M$ and $N$ to be the exotic pairs of Mazur manifolds $W_n, W_n'$ from \cite{menagerie}, which can be easily checked to glue to $S^4$. It is likely that the exotic pairs from \cite{Gheehyun} also glue to $S^4$, but we have not verified that.
However, it may also happen that there is some diffeomorphism $M \cong N$, but it does not restrict to $\theta$ on the boundary. In other words, we only require that the pair $(M, N)$ be exotic relative to $\theta$. Such examples include the Akbulut cork \cite{Akbulut_cork} and the generalized family of corks from \cite{Akbulut-Yasui}.
\end{remark}

\begin{remark}
Note that the pair $(M', N')$ is a \textit{multicork} in the language of \cite{Melvin-Schwartz}. Since we have $M' \cup_{\theta'} -N' = S^4$, it is a \textit{simple} multicork. Finally, since $\theta'$ does not extend to a diffeomorphism $\Theta'$, we are justified in calling it a \textit{nontrivial} multicork (see \cite{Equivariant} for a definition of nontrivial corks).    
\end{remark}

\subsection{Beyond Mazur manifolds}\label{1.4}
We conclude by observing that most of our results can be extended to general contractible AC 4-manifolds, which can have multiple homotopically canceling 1/2-handle pairs. One only needs to repeat our arguments for each of these pairs.

First of all, given such manifolds $X, X'$ in the filtration level $\A^k$ (which means there are $k$ 1/2-handle pairs, see Section~\ref{2}), one can again call a pair $(X, X')$ \textit{ribbon} if there exists a ribbon concordance $C 
\subset \#^k S^1 \times S^2 \times I$ from the framed link determining the 2-handles of $X$ to that of $X'$. Surgering $C$, we obtain a ribbon homology cobordism $W_C$ from $\partial X$ to $\partial X'$. The arguments of Section~\ref{3} go through almost verbatim to show that $X \cup W_C \cup -X' \cong S^4$, which defines a link $L_C : X \sqcup -X' \hookrightarrow S^4$. Just like in Theorem~\ref{thm1}, this link has smoothly standard components but is not split because $\pi_1(W_C) \neq 1$. 

Similarly, one can replace the $M$ and $M'$ of Theorem~\ref{thm4} by such $X$ and $X'$, where we must again assume that $X$ admits an exotic embedding in $S^4$ with AC complement. Theorem~\ref{thm3} could also be generalized to this setting if one adjusts the assumptions appropriately.

Theorem~\ref{thm2} is the only result where it is essential for the manifolds to be Mazur. Given a ribbon pair $(X, X')$ with $X, X' \in \A^k$, we may still \enquote{blow up} the links $L_C$ to obtain links in $\#^n \BC \BP^2$. The first issue is that one now needs at least $k$ blow-ups to ensure that the link complement is simply connected, i.e. $|n| \geq k$. The second, more important, issue is that we can no longer appeal to \cite{Conway-Tosun} to show that $\partial X$ is not an L-space. Thus, to obtain exotically split links one would need to assume, additionally, that $\partial X$ is not an L-space.\\

\textbf{Conventions.} Given an oriented manifold $X$, we will denote by $-X$ the same manifold equipped with the opposite orientation. All maps will be orientation-preserving unless specified otherwise. As usual, we denote a homotopy equivalence by $\simeq$, a homeomorphism by $\approx$, and a diffeomorphism, as well as a group isomorphism, by $\cong$. The more canonical equivalences will be denoted by $=$.

All pictures will represent examples, rather than rigorous expositions. In light of that, we will always draw the Akbulut cork \cite{Akbulut_cork} in place of a general Mazur manifold for simplicity. All 1-handles are oriented counterclockwise, while the orientation of the 2-handles depends on the context.

The term \textit{exotic} will always refer to structures and maps that are standard or equal topologically but not smoothly. Since we discuss different types of exotica, the exact meaning of the term will depend on the context. The 4-manifolds $S^4$ and $\#^n \BC \BP^2$ are endowed with their standard smooth structures. \\

\textbf{Organization.} In Section~\ref{2} we discuss what is known about codimension-zero embeddings in 4D and introduce some basic techniques. In Section~\ref{3} we describe our construction of the link $L_C$ and prove Theorem~\ref{thm1}. In Section~\ref{4} we construct the links $L_{C, n}$ and prove Theorem~\ref{thm2}. In Sections~\ref{5} and~\ref{6} we prove Theorems~\ref{thm3} and~\ref{thm4}, respectively. Finally, in Section~\ref{7} we discuss further questions along with some limitations.\\

\textbf{Acknowledgments.} We thank Francesco Lin for his continuous help and support, as well as for suggesting Question~\ref{Q}, which led to this project. We thank Tye Lidman and Marco Marengon for very insightful discussions. We also thank Maya Chande, Anthony Conway, David Gabai, Kyle Hayden, Juan Muñoz-Echániz, Mark Powell, and Bülent Tosun for helpful comments. The author was partially supported by NSF grant DMS-2503714.

\section{Knotting contractible 4-manifolds in dimension four}\label{2}
In this section, we discuss smooth, contractible 4-manifolds and their embeddings in $S^4$ as well as larger 4-manifolds. An arbitrary contractible 4-manifold is not easy to work with, so we usually restrict our attention to those with particularly simple handle structures.

Let $\C$ denote the family of all smooth, compact, oriented, contractible 4-manifolds. Let $\A \subset \C$ denote the subfamily of contractible AC manifolds, as defined in \cite{Melvin-Schwartz}. A manifold $A \in \A$ is built by attaching $k$ 1-handles $x_i$ to $B^4$, followed by $k$ 2-handles $r_i$ attached in a way that results in the standard presentation of the trivial group: 
$$\pi_1(A) = \langle x_1, x_2, \dots, x_k \mid r_1, r_2, \dots, r_k \rangle = 1, \quad r_i = x_i.$$

The acronym AC refers to the Andrews--Curtis moves on group presentations \cite{AC} in view of the following observation. Let $X$ be a contractible 4-manifold built from $k$ 1-handles $x_i$ and $k$ 2-handles $y_i$ so that the induced presentation of the trivial group
$$1 = \pi_1(X) = \langle x_1, x_2, \dots, x_k \mid y_1, y_2, \dots, y_k \rangle$$
is Andrews--Curtis equivalent to the trivial presentation. Then the Andrews--Curtis moves provide a guide on how to slide the 2-handles $y_i$ so as to obtain new 2-handles $r_i$ inducing the standard presentation of the trivial group. This implies that $X \in \A$. 

One may enlarge $\A$ to a family $\B \subset \C$ consisting of contractible manifolds given by an equal number of 1- and 2-handles, without any assumption on the presentation of their trivial fundamental group. If the Andrews--Curtis conjecture \cite{AC} were true, we would have $\A = \B$. We also point out that to our best knowledge, there are no examples of manifolds $C \in \C$ that are not in $\A$. 

There is a natural filtration on $\A$:
$$\A^1 \subset \A^2 \subset \cdots \subset \A^n \subset \cdots,$$
where $\A^n$ is the subfamily admitting a handle decomposition with at most $n$ pairs of 1- and 2-handles.

We are especially interested in the contractible 4-manifolds in the lowest filtration level. In accordance with the literature \cite{AkbulutKirby}, we call these manifolds \textit{Mazur} and denote $\M \coloneq \A^1$. By calling an abstract 4-manifold Mazur, we implicitly assume it is equipped with a handle decomposition that has a single 2-handle.
Such manifolds $M \in \M$ are uniquely determined by their 2-handle, which is given by a framed knot $K \subset S^1 \times S^2$ with winding number $1$, allowing us to write $M = M(K)$. 

We consider a further specialization by letting $\M_0 \subset \M$ denote the subfamily of those Mazur manifolds whose 2-handles are 0-framed and unknotted. Recall, that given $M_0 \in \M_0$, one can perform a zero-dot exchange to obtain a new Mazur manifold $M^{\bullet} \in \M_0$, together with a boundary diffeomorphism $\theta : \partial M_0 \xrightarrow{\cong} \partial M^{\bullet}$. 

We define one final subfamily $\M_S \subset \M_0$ (compare with the \textit{reasonably nice} Mazur manifolds in \cite{menagerie}). Assume a Mazur manifold $M_0 = M(K_0) \in \M_0$ is given by a knot $K_0 \subset S^1 \times S^2$, and recall that $\on{tb}(M_0)$ denotes the maximal Thurston--Bennequin invariant of a Legendrian realization of $K_0$ in the standard contact structure of $S^1 \times S^2$. Let $M_0 \in \M_S$ if $\on{tb}(M_0) > 0$, which implies that $M_0$ admits a Stein structure by the Eliashberg criterion. 

It can be shown, using the methods of \cite{Akbulut-Yasui}, that for $M_0 \in \M_S$, the boundary map $\theta : \partial M_0 \rightarrow \partial M^{\bullet}$ does not extend to a diffeomorphism $\Theta : M_0 \xrightarrow{\cong} M^{\bullet}$. Note that $\theta$ must always extend to a homeomorphism $M_0 \xrightarrow{\approx} M^{\bullet}$ by Freedman's work \cite{Freedman}. Thus, the family $\M_S$ consists of Mazur manifolds with some exotic property. 

Note that unlike the larger families, $\M_S$ is not closed under orientation reversal, i.e., $M_0 \in \M_S$ is not equivalent to $-M_0 \in \M_S$. Nevertheless, given $M_0 \in \M_S$, it still follows that the boundary map $-\theta : -\partial M_0 \rightarrow - \partial M^{\bullet}$ does not extend to a diffeomorphism $-\Theta : -M_0 \xrightarrow{\cong} -M^{\bullet}$ since that would also give a diffeomorphism $\Theta : M_0 \xrightarrow{\cong} M^{\bullet}$ extending $\theta$.

We have inclusions $$\M_S \subset \M_0 \subset \M \subset \A \subset \B \subset \C.$$

\subsection{Knots in \texorpdfstring{$S^4$}{S4}}\label{2.1}

First of all, we observe that it makes sense to speak about knotting a contractible 4-manifold $C \in \C$ only after we fix a \enquote{standard} embedding $C \hookrightarrow S^4$, which might not even exist. 

On the other hand, every 4-manifold $X \in \A$ embeds in $S^4$, and in a fairly standard way. Let $\Sigma = X \cup -X$ be the double of $X$ and view it as the boundary of $X \times I$, which inherits a handle structure from $X$. By definition, each 2-handle $r_i$ homotopically cancels the 1-handle $x_i$ in the handle structure of $X$. Since homotopy implies isotopy for curves in a 4-manifold, the corresponding handles of $X \times I$ must geometrically cancel, giving $X \times I \cong B^5$. It follows that 
$$\Sigma = X \cup -X =  \partial (X \times I) \cong \partial B^5 = S^4,$$ providing an embedding $X \hookrightarrow S^4$ with complement $S^4 \backslash X = -X$. This symmetry is what allows us to call this embedding \textit{standard}, and we shall denote it $\varphi_X$. See \cite{Melvin-Schwartz} for more details.

We note that there is also an alternative, more explicit demonstration of the above fact, which relies on Kirby calculus. Its main idea is used repeatedly in our proof of Theorem~\ref{thm4}.

We also note that \textit{a priori,} the embedding $\varphi_X$ depends on the AC structure of $X$. Indeed, two different AC structures yield diffeomorphisms $\phi_1, \phi_2 : X \cup -X \rightarrow S^4$, which may not be smoothly isotopic in view of the existence of an exotic diffeomorphism of $S^4$ \cite{GGH}. However, it turns out that even if $\phi_i$ are not isotopic, the two AC structures induce the same smooth embedding $X \hookrightarrow S^4$ for the following reason. 

Let $\varphi_1, \varphi_2 : X \hookrightarrow S^4$ be the embeddings obtained by composing the map $X \hookrightarrow X \cup -X$ with the diffeomorphisms $\phi_1$ and $\phi_2$, respectively. Then we have a diffeomorphism of pairs
$$\phi_1 \circ \phi_2^{-1} : (S^4, \varphi_2(X)) \rightarrow (S^4, \varphi_1(X)).$$
As we shall see shortly, such smooth equivalences in $S^4$ lead to smooth isotopies. Therefore, $\varphi_1$ and $\varphi_2$ are isotopic, and hence the standard embedding $\varphi_X$ does not actually depend on the AC structure of $X$. 

\begin{defn}\label{defn}
Given $X \in \A$ and a smooth (topological) embedding $\varphi : X \hookrightarrow S^4$, we say $\varphi$ is smoothly (topologically) standard if there exists a smooth (topological) isotopy $\varphi \sim \varphi_X$, where $\varphi_X$ is determined by some AC structure on $X$.
\end{defn}

Even though the embedding $\varphi_X$ does not depend on the particular AC structure on $X$, one can ask the following question, which we do not pursue.

\begin{quest}
Let $X \in \A$, and let $\phi_1, \phi_2 : X \cup -X \rightarrow S^4$ be diffeomorphisms induced by two AC structures on $X$. Is $\phi_1 \circ \phi_2^{-1} : S^4 \rightarrow S^4$ smoothly isotopic to the identity map?
\end{quest}

Having chosen standard embeddings for $\A$, we are now free to discuss \textit{knotted}, i.e., non-isotopic, embeddings. The main result in this direction is the following theorem by Lickorish, where a balanced presentation of a group is one with an equal number of generators and relations.

\begin{thm}[{\cite{Lickorish}}]\label{Lickorish}
    Let $G$ be any finitely presented, perfect group with a balanced presentation. There exists a 4-manifold $X_G \in \A$ and an embedding $\varphi : X_G \hookrightarrow S^4$ such that $\pi_1(S^4 \backslash \varphi(X_G)) = G$. 
\end{thm} 

Livingston makes several observations about Lickorish's construction in \cite{Livingston}.
First, he shows that for each group $G$ as above, one can find infinitely many manifolds $X_G$ that can be knotted in $S^4$ with fundamental group $G$. Moreover, one can select these manifolds to be distinct for different groups $G$. The idea is to take connected sums with the standard embedding of some Mazur manifold $M$, which does not affect the fundamental group of the complement.

It is evident in Lickorish's proof that if the group $G$ has a balanced presentation of size $n$, then $X_G$ can be chosen from the filtration level $\A^n$. It is natural to ask when the manifolds $X_G$ can be chosen from the lower filtration levels. Livingston's second observation states that if $G$ can be obtained by adding a single relation to a classical knot group $\pi_1(S^3 \backslash K)$, then the corresponding manifold $X_G$ can be chosen to be Mazur: $X_G \in \A^1 = \M$. 

Finally, Livingston shows that if $G$ surjects onto a nontrivial finite quotient of a 2-knot group, then there exist infinitely many non-isotopic embeddings $\phi_i : X_G \hookrightarrow S^4$ distinguished by $\pi_1(S^4 \backslash \phi_i(X_G))$. In a later paper \cite{Livingston2}, Livingston also finds embeddings of contractible 4-manifolds realizing groups with arbitrarily large deficiency (groups with a balanced presentation have deficiency $0$).

One shortfall of these results is that they work in one direction: one starts with a group $G$, and based on its presentation or other properties, one constructs a contractible 4-manifold $X_G$ that can be knotted with group $G$. However, there is no strategy for starting with a prescribed manifold and finding an embedding with a nontrivial group. This leads to the following natural question.

\begin{quest}
Given an arbitrary 4-manifold $X \in \A$, does there exist an embedding $\varphi : X \hookrightarrow S^4$ such that $\pi_1(S^4 \backslash \varphi(X)) \neq 1$? Can such an embedding be found for Mazur manifolds $M \in \M$?
\end{quest}

It turns out that if we restrict our attention to the topological category, then the standard embeddings are detected by the fundamental group of the complement.
\begin{prop}\label{topknot}
    For $X \in \A$, if a topological embedding $\varphi : X \hookrightarrow S^4$ satisfies $\pi_1(S^4 \backslash \varphi(X)) = 1$, then it is topologically isotopic to the standard embedding $\varphi_X$. 
\end{prop}

\begin{proof}
Assume $\varphi$ is as above; then $-X$ and $S^4 \backslash \varphi(X)$ are two contractible, hence simply connected, topological 4-manifolds with homeomorphic boundaries. Since $Y = \partial X$ must be a homology sphere, it follows from \cite{Freedman} that there exists a homeomorphism $\phi : S^4 \backslash \varphi(X) \rightarrow -X$ extending the identity map on $Y$. Further extending $\phi$ across $X$ by identity, one obtains a homeomorphism of pairs 
$$\phi_0 : (S^4, \varphi(X)) \xrightarrow{\approx} (S^4, \varphi_X(X)).$$ 
Next, choose a 4-ball $B \subset S^4$ containing $\varphi(X)$. We can find an isotopy $\phi_t$ from $\phi_0$ to a homeomorphism $\phi_1 : S^4 \rightarrow S^4$ that fixes $B$. Then it must also fix $\varphi(X) \subset B$, so we have $$\phi_0(\varphi(X)) = \varphi_X(X),\quad \phi_1(\varphi(X)) = \varphi(X).$$ It follows that $\phi_t \circ \varphi$ is a topological isotopy between the embeddings $\varphi$ and $\varphi_X$, as needed. 
We learned this strategy for turning an equivalence into an isotopy from \cite{Budney}.
\end{proof}

More generally, given $X \in \A$, one may wish to classify all embeddings $X \hookrightarrow S^4$ up to topological isotopy.
\begin{quest}
    Under what conditions are two embeddings $\varphi, \varphi' : X \hookrightarrow S^4$ topologically isotopic?
\end{quest}

Of course, one must require that $\pi_1(S^4 \backslash \varphi(X)) \cong \pi_1(S^4 \backslash \varphi'(X))$. However, one needs additional algebraic conditions when these groups are non-trivial; see \cite{Conway-Kasprowski} for details.\\

Moving beyond the fundamental group, note that the embeddings discussed so far are knotted both smoothly and topologically. This is because their complements are not simply connected, and hence cannot be homeomorphic to the complements of the standard embeddings, which are contractible. It is natural to ask whether there exist \textit{exotically} knotted embeddings $X \hookrightarrow S^4$, which are topologically but not smoothly isotopic to the standard embedding. Of course, in order to detect such embeddings one would need smooth obstructions rather than algebraic ones. 

Thus, we are interested in when a topologically standard embedding $\varphi$ can be shown to be smoothly non-isotopic to the standard embedding. Let $M_0 \in \M_0$ and recall that the zero-dot exchange operation gives rise to an embedding $\varphi : M_0 \hookrightarrow S^4$ with complement $S^4 \backslash \varphi(M_0) = -M^{\bullet}$. Since $M^{\bullet}$ is contractible, $\varphi$ must be topologically isotopic to $\varphi_{M_0}$. On the other hand, these embeddings can be smoothly distinguished for the subfamily $\M_S \subset \M_0$. 

Assume $M_0 \in \M_S$ and consider the decomposition $S^4 = M_0 \cup_{\theta} -M^{\bullet}$ induced by $\varphi$. 
A smooth isotopy $\varphi \sim \varphi_{M_0}$ would give rise to a smooth equivalence between the decompositions $M_0 \cup_{\theta} - M^{\bullet}$ and $M_0 \cup_{\Id} -M_0$ extending the identity map on $M_0$. But that would provide a diffeomorphism $\Theta : M_0 \xrightarrow{\cong} M^{\bullet}$ extending $\theta$, resulting in a contradiction. Thus, $\varphi : M_0 \hookrightarrow S^4$ must be an exotically knotted embedding. 

This connection between exotic embeddings and the non-extension of a boundary map is more universal, as we hinted at in the introduction. Let $X \in \A$ and assume $\varphi : X \hookrightarrow S^4$ is topologically standard. Then $Y \coloneq -S^4 \backslash \varphi(X)$ must be homeomorphic to $X$. Let $\theta : \partial X \rightarrow \partial Y$ be the obvious boundary diffeomorphism. If $\varphi$ were smoothly standard, $\theta$ would clearly extend to a diffeomorphism $\Theta : X \xrightarrow{\cong} Y$. Conversely, if $\theta$ extends to such $\Theta$, we can further extend $\Theta$ over $X$ by identity, resulting in a smooth equivalence between $\varphi$ and $\varphi_X$. This leads to a smooth isotopy $\varphi \sim \varphi_X$ as in the proof of Proposition~\ref{topknot}. We obtain the following.
\begin{prop}\label{propEx}
A smooth embedding $\varphi : X \hookrightarrow S^4$ with complement $-Y$ is smoothly isotopic to $\varphi_X$ if and only if $\theta : \partial X \rightarrow \partial Y$ extends to a diffeomorphism $\Theta : X \xrightarrow{\cong} Y$.
\end{prop}

\begin{remark}\label{rem2.6}
A similar argument shows that if $\theta$ does not extend, then the restriction of $\varphi$ to $\partial M$ is also topologically but not smoothly isotopic to the restriction of $\varphi_{M}$ to $\partial M$. Expanding on this idea, Auckly and Ruberman in \cite{Auckly-Ruberman} use the infinite-order corks constructed by Gompf \cite{Gompf} to find an infinite family of homology 3-spheres, and for each of them an infinite family of topologically isotopic but smoothly non-isotopic embeddings into $S^4$. 
\end{remark} 

\subsection{Links in \texorpdfstring{$S^4$}{S4}}\label{2.2}
Having gotten a taste of knotting contractible 4-manifolds in $S^4$, we turn our attention to the question of \textit{links} of several contractible 4-manifolds given as $$L : X_1 \sqcup \dots \sqcup X_m \hookrightarrow S^4.$$ Since $L$ restricts to an embedding of each of the components $X_i$, it makes sense to restrict our attention to the case $X_i \in \A$, so that we have the standard embeddings $X_i \hookrightarrow S^4$. 
Factoring these through a 4-ball $B^4 \subset S^4$, we obtain standard embeddings $X_i \hookrightarrow B^4$, which, after a choice of disjoint 4-balls $B_i \subset S^4$, lead to the definition of the \textit{unlink} $$L_0 : X_1 \sqcup \dots \sqcup X_m \hookrightarrow B_1 \sqcup \dots \sqcup B_m \subset S^4.$$ Since the standard embeddings $X_i \hookrightarrow B_i$ satisfy $\pi_1(B_i \backslash X_i) = \pi_1(-X_i \backslash B^4) = 1$, we obtain, by inductively applying the van Kampen theorem, $\pi_1(S^4 \backslash L_0) = 1$. 

\begin{remark}
We briefly note that the decomposition $S^4 = X \cup -X$ provides a link
$$L_X :  X \sqcup - X \hookrightarrow S^4$$
after pushing each copy of $X$ inside. This link has complement $\partial X \times I$, which actually implies that it cannot be topologically split by $S^3$. This is because any homology 3-sphere $Y$ splitting $L$ embeds into $\partial X \times I$ as a homological section. Composing with the projection $\partial X \times I \rightarrow \partial X$, we obtain a map $Y \rightarrow \partial X$ of degree one. Then the induced map on the fundamental group must be a surjection $\pi_1(Y)  \twoheadrightarrow \pi_1(\partial X)$, hence $Y$ is not $S^3$. This idea comes from \cite{Marcos}. Since each $X \in \A$ gives rise to a non-split link $L_X : X \sqcup -X \hookrightarrow S^4$, the more interesting question is to find non-split links of distinct Mazur manifolds, which is what Theorem~\ref{thm1} accomplishes. Lastly, note that when $C$ is the trivial concordance, our link $L_C$ agrees with $L_X$. 
\end{remark}

We wish to compare a general link $L : X_1 \sqcup \dots \sqcup X_m \hookrightarrow S^4$ to $L_0$ in the topological category, as we did for knots $X \hookrightarrow S^4$. Clearly, the two links would be distinct if their restrictions to any of the $X_i$ were different. Thus, we must assume that $L$ induces a topologically standard embedding of each $X_i$, or equivalently, $\pi_1(S^4 \backslash L
(X_i)) = 1$. Similarly, if $\pi_1(S^4 \backslash L) \neq 1$, then $L$ is not even topologically equivalent to $L_0$, i.e., there is no homeomorphism between the pairs $(S^4, L_0)$ and $(S^4, L)$. Thus, we must also assume $\pi_1(S^4 \backslash L) = 1$, which actually implies $\pi_1(S^4 \backslash L(X_i)) = 1$. We claim that this condition is sufficient.

\begin{prop}\label{toplink}
Let $L : X_1 \sqcup \dots \sqcup X_m \hookrightarrow S^4$ be a topological link satisfying $\pi_1(S^4 \backslash L) = 1$. Then $L$ is topologically isotopic to the unlink $L_0$. 
\end{prop}

\begin{proof}
First, we confirm that the assumption $\pi_1(S^4 \backslash L) = 1$ indeed implies $\pi_1(S^4 \backslash L(X_i)) = 1$. Starting with $S^4 \backslash L$, we can inductively glue the $X_i$ back in until we obtain $S^4 \backslash L(X_m)$. After each gluing, an application of the van Kampen theorem shows that the new space remains simply connected, hence $\pi_1(S^4 \backslash L(X_m)) = 1$. Similarly, $\pi_1(S^4 \backslash L(X_i)) = 1$ for all $i$. It follows from Proposition~\ref{topknot} that every restriction $L|_{X_i}$ is topologically standard. 

Next, we wish to find a homeomorphism of pairs $\Phi_0 : (S^4, L) \xrightarrow{\approx} (S^4, L_0)$. Comparing the complements $W \coloneq S^4 \backslash L$ and $W_0 \coloneq S^4 \backslash L_0$, we see that they are both simply connected, have trivial second homology, and have homeomorphic boundaries. However, Freedman's result does not apply directly since their boundary consists of multiple connected components. We reduce to the case of connected boundary following \cite[Construction~3.5]{Orson-Powell}.

For each $1 \leq i < m$, choose properly embedded, disjoint arcs $\gamma_i \subset W_0$ and $\gamma_i' \subset W$ so that they have one end on $Y_i = \partial X_i$ and another on $Y_{i+1} = \partial X_{i+1}$. We can assume that the endpoints of $\gamma_i$ and $\gamma_i'$ agree in $Y_i$ and $Y_{i+1}$. Deleting these arcs, we obtain
$$W' \coloneq W \backslash \cup_i \nu(\gamma_i'), \quad W_0' \coloneq W_0 \backslash \cup_i \nu(\gamma_i)$$
with boundary
$$Y \coloneq \partial W' = Y_1 \# \cdots \#Y_m = \partial W_0'.$$
Since $W'$ and $W_0'$ are 4-manifolds with connected boundary, Freedman's result applies, giving us a homeomorphism of pairs 
$$\Phi_0'' : (W', Y) \xrightarrow{\approx} (W_0', Y).$$

Since we chose the arcs $\gamma_i$ and $\gamma_i'$ so that their endpoints are identified, $\Phi_0''$ extends across the neighborhoods $\nu(\gamma_i) \cong \gamma_i \times B^3$ and $\nu(\gamma_i') \cong \gamma_i' \times B^3$ by identity, resulting in a homeomorphism
$$\Phi_0' : W \xrightarrow{\approx} W_0, \quad \Phi_0' \coloneq (\Phi_0'', \Id) : W' \cup_i \nu(\gamma_i') \xrightarrow{\approx} W_0' \cup_i \nu(\gamma_i).$$
Finally, we extend $\Phi_0'$ across each $X_i$ by identity, obtaining a homeomorphism of pairs
$$\Phi_0 : (S^4, L) \xrightarrow{\approx} (S^4, L_0).$$

This implies that the links $L$ and $L_0$ are topologically equivalent. To obtain a topological isotopy, we consider the 4-balls $B_i \subset S^4$ containing $L_0(X_i)$ and their preimages $B_i' \coloneq \Phi_0^{-1}(B_i) \subset S^4$, which contain $L(X_i)$. We can find an isotopy $\Phi_t$ from $\Phi_0$ to a homeomorphism $\Phi_1 : S^4 \xrightarrow{\approx} S^4$ which restricts to identity on each $B_i'$. In particular, $\Phi_1$ fixes $L(X_i) \subset B_i'$ and therefore $\Phi_1(L) = L$. Since $\Phi_0(L) = L_0$, $\Phi_t \circ L$ provides a topological isotopy between the links $L_0$ and $L$, as desired. 
\end{proof}

Turning to the question of exotic links, we note that smoothly distinguishing two topologically isotopic links $L, L_0$ is difficult. Of course, this can be done if their restrictions to some $X_i$ are smoothly non-isotopic, but this simply reduces the question to distinguishing smooth \textit{knots} of contractible 4-manifolds. Indeed, to the best of our knowledge, Theorem~\ref{thm2} provides the first example of topologically isotopic but smoothly non-isotopic links of contractible 4-manifolds that are not distinguished by individual components. However, this is achieved not in $S^4$, but only in definite 4-manifolds (we discuss links in larger 4-manifolds below).

\subsection{Knots in larger 4-manifolds}\label{2.3}
Let $Z$ be a smooth, compact, connected, and oriented 4-manifold other than $S^4$, such as $\#^n \BC \BP^2$. Given a contractible manifold $X \in \A$, we obtain a \textit{standard} embedding $\varphi_Z : X \hookrightarrow Z$ by using $\varphi_X$ to embed $X$ in a 4-ball $B^4 \subset Z$, which is unique by \cite{Palais}. We turn to the question of \textit{knotted} embeddings $\varphi : X \hookrightarrow Z$ that are not isotopic to $\varphi_Z$.

It turns out that the topologically standard embeddings are essentially detected by the fundamental group of the complement when $Z$ is simply connected, similarly to the case of $S^4$.
\begin{prop}\label{bigknot}
Let $Z$ be a smooth, closed, simply connected 4-manifold and let $X \in \A$. If a smooth embedding $\varphi : X \hookrightarrow Z$ satisfies $\pi_1(Z \backslash \varphi(X)) = 1$, then it is topologically isotopic to $\varphi_Z$. 
\end{prop}

\begin{proof}
    Set $Y = \partial X$ and consider the complements $Z_1 = Z \backslash \varphi(X)$ and $Z_0 = Z \backslash \varphi_Z(X)$. They are both simply connected, have boundary $Y$ which is a homology 3-sphere, and must have intersection forms isomorphic to $Q_Z$ since $X$ is contractible. It follows from \cite{Boyer1} that there is a homeomorphism $\phi : Z_1 \rightarrow Z_0$ which restricts to identity on $Y$. Extending to $X$ by identity, we obtain a homeomorphism of pairs
    $$\phi_0 : (Z, \varphi(X)) \xrightarrow{\approx} (Z, \varphi_Z(X)).$$
    Since $\varphi_Z(X)$ lies in a 4-ball, we can pull it back via $\phi_0$ to obtain a topological 4-ball containing $\varphi(X)$. Taking an isotopy $\phi_t$ to a homeomorphism $\phi_1$ fixing that 4-ball, we obtain a topological isotopy $\phi_t \circ \varphi$ between $\varphi_Z$ and $\varphi$. Here we only needed $\varphi$ to be smooth in order to avoid dealing with the Kirby--Siebenmann invariant in \cite[Theorem~0.7]{Boyer1}.
\end{proof}

Next, we extend Proposition~\ref{propEx} to this setting.

\begin{prop}\label{propbigEx}
For $X \in \A$ and a smooth embedding $\varphi : X \hookrightarrow Z$, $\varphi$ is smoothly isotopic to $\varphi_Z$ if and only if there exists a diffeomorphism $Z \backslash \varphi(X) \cong Z \backslash \varphi_Z(X)$ extending the identity map on $\partial X$. 
\end{prop}

\begin{proof}
If $\varphi$ is smoothly isotopic to $\varphi_Z$, then such a diffeomorphism clearly exists. Conversely, assume there is a diffeomorphism $\Theta : Z \backslash \varphi(X) \rightarrow Z \backslash \varphi_Z(X)$ that restricts to identity on $\partial X$. Extending across $X$ by identity, we obtain a diffeomorphism of pairs $\Theta' : (Z, \varphi(X)) \rightarrow (Z, \varphi_Z(X))$, therefore the two embeddings are smoothly equivalent. Finally, since $\varphi_Z$ is standard, $\varphi_Z(X)$ must be contained in a smooth 4-ball $B \subset Z$, and thus we can obtain a smooth isotopy $\varphi \sim \varphi_Z$ just like before.
\end{proof}

It is natural to distinguish between locally and globally knotted embeddings of $X$. We call an embedding $\varphi : X \hookrightarrow Z$ \textit{local} if its image can be isotoped to lie inside a 4-ball, and we call it \textit{global} otherwise. Understanding local embeddings $X \hookrightarrow Z$ reduces to the question of knots in $S^4$, hence the right question to study for larger 4-manifolds is about globally knotted embeddings $\varphi : X \hookrightarrow Z$ that cannot be isotoped to lie inside a 4-ball. 

It turns out that the distinction between local and global knotting phenomena is not well understood even for 2-spheres. Indeed, the author is aware only of the following result by Lee.

\begin{thm}[\cite{Lee}]\label{propLee}
There exist smooth, null-homotopic spheres $K \hookrightarrow S^2 \times S^2$ that cannot be isotoped to lie inside a 4-ball, even topologically. 
\end{thm}

The obstruction simply uses the fundamental group: Lee constructs the spheres $K$ so that $H_2(\pi_1(S^2 \times S^2 \backslash K)) \neq 0$, whereas a short homological argument shows that local knots must have trivial second homology of the fundamental group of the complement. 

Perhaps somewhat surprisingly, the same obstruction works for embeddings of contractible 4-manifolds. In particular, Lee's ideas and Theorem~\ref{Lickorish} readily give the following analog. Let $Z$ be either $S^2 \times S^2$ or $\BC \BP^2 \# \overline{\BC \BP^2}$, and let $G$ be a perfect group with a balanced presentation such that the intersection $H$ of its center and commutator subgroup is nontrivial. 

\begin{prop}\label{propC}
There exists a 4-manifold $X_G \in \A$, and for each $h \in H$ an embedding $\varphi_h : X_G \hookrightarrow Z$ with $\pi_1(Z \backslash \varphi_h(X_G)) = G/\langle h \rangle$. It follows that $X_G$ cannot be isotoped to lie inside any contractible 4-manifold $C \subset Z$ satisfying $\pi_1(Z \backslash C) = 1$.
\end{prop}
One simply begins with an embedding $\varphi' : X_G \hookrightarrow S^4$ with $\pi_1(S^4 \backslash \varphi'(X_G)) = G$ coming from Theorem~\ref{Lickorish}, and proceeds as in \cite{Lee} to obtain the required embeddings $\varphi_h : X_G \hookrightarrow Z$. It can be checked that Lee's idea also works for $\BC \BP^2 \# \overline{\BC \BP^2}$, and that the obstruction applies to every such $C$, not only the 4-ball.
When $G$ is obtained from a knot group $\pi_1(S^3 \backslash K)$ by adding a single relation, $X_G$ can be chosen to be Mazur, as in \cite{Livingston}. Similarly, one may apply Livingston's other observations to this setting. 

Note that these embeddings are global both topologically and smoothly. One can ask for embeddings $X \hookrightarrow Z$ that are topologically local but smoothly global. The construction of \cite{Dehn_twists} provides such examples where $Z$ is closed and indefinite. Our Corollary~\ref{cor2} provides such embeddings for every $X \in \M$, where the ambient manifold $Z$ is definite but not closed. We conclude with the following question.

\begin{quest}
    Do there exist a closed, definite 4-manifold $Z$ and an embedding of a contractible 4-manifold $X \hookrightarrow Z$ that is topologically global? What about an embedding that is topologically local but smoothly global?
\end{quest}

\subsection{Links in larger 4-manifolds}\label{2.4}

Given a smooth, compact, connected, and oriented 4-manifold $Z$ and 4-manifolds $X_i \in \A$, we can again choose disjoint 4-balls $B_i \subset Z$ and define an \textit{unlink}
$$L_0 : X_1 \sqcup \dots \sqcup X_m \hookrightarrow B_1 \sqcup \dots \sqcup B_m \subset Z.$$
As one might expect, the topological unlink is once again detected by the fundamental group when $Z$ is simply connected.
\begin{prop}\label{biglink}
Let $Z$ be a smooth, closed, simply connected 4-manifold and let $X_i \in \A$. If a smooth link $L : X_1 \sqcup \cdots \sqcup X_m \hookrightarrow Z$ satisfies $\pi_1(Z \backslash L) = 1$, then it is topologically isotopic to the unlink $L_0$. 
\end{prop}

\begin{proof}
We simply combine the arguments from Propositions~\ref{toplink} and~\ref{bigknot}.
Since $Z$ is simply connected, the condition $\pi_1(Z \backslash L) = 1$ implies that $\pi_1(Z \backslash L(X_i)) = 1$ for each $i$, as in Proposition~\ref{toplink}. It then follows from Proposition~\ref{bigknot} that each restriction $L|_{X_i}$ is topologically standard. 

We compare the complements $Z_1 = Z \backslash L$ and $Z_0 = Z \backslash L_0$. As in the case of $S^4$, we delete $m-1$ properly embedded arcs $\gamma_i \subset Z_0$ and $\gamma_i' \subset Z_1$, obtaining simply connected, smooth 4-manifolds $Z_0'$ and $Z_1'$ with connected boundary $Y = \#_i (\partial X_i)$. Note that deleting these arcs does not affect the intersection form; thus $Z_0'$ and $Z_1'$ still have intersection forms isomorphic to the intersection form of $Z$. This allows us to apply Boyer's theorem to find a homeomorphism 
$$\Phi_0' : (Z_1', Y) \xrightarrow{\approx} (Z_0', Y),$$
extending the identity map on $Y$. Gluing the neighborhoods $\nu(\gamma_i), \nu(\gamma_i')$ and the $X_i$ back in, and extending $\Phi_0'$ by identity, we obtain a homeomorphism of pairs
$$\Phi_0 : (Z, L) \xrightarrow{\approx} (Z, L_0).$$ 

Finally, we upgrade the topological equivalence to a topological isotopy using the same strategy. Choose 4-balls $B_i \subset Z$ containing $L_0(X_i)$, pull them back using $\Phi_0$ to obtain 4-balls $B_i' \subset Z$ containing $L(X_i)$, and isotope $\Phi_0$ to a homeomorphism $\Phi_1$ that fixes each $B_i'$. Then $\Phi_t \circ L$ is a topological isotopy from $L_0$ to $L$.
\end{proof}

\section{2-component links in \texorpdfstring{$S^4$}{S4}}\label{3}
We begin by unpacking Definition~\ref{ribbon}. The Kirby diagram of $S^1 \times B^3$ consisting of a dotted unknot provides a surgery diagram for $S^1 \times S^2$. A Mazur manifold $M$ is uniquely determined by its 2-handle, which is attached along a framed knot $K \subset S^1 \times S^2$ with winding number 1.

Assume we have a Mazur manifold $M \in \M$ given by such a knot $K$, and a ribbon concordance $C \subset S^1 \times S^2 \times I$ from $K \subset S^1 \times S^2 \times \{0\}$ to some knot $K' \subset S^1 \times S^2 \times \{1\}$. Since $K$ and $K'$ are concordant, they must be freely homotopic; thus $K'$ also has winding number 1. 
The neighborhood of $C \cong S^1 \times I$ is diffeomorphic to $\nu(C) \cong S^1 \times B^3$; however, this diffeomorphism is not unique and corresponds to a trivialization of the normal bundle of $C$, which is described by an integer. 
Choose the unique trivialization of the normal bundle of $C$ inducing the correct framing on $K$. This uniquely determines a framing of the knot $K'$, which is what we mean when we say $C$ preserves the framing in Definition~\ref{ribbon}.

Such a concordance $C$ is \textit{ribbon} if the projection $S^1 \times S^2 \times I \rightarrow I$ restricts to a Morse function on $C$ with no index-2 critical points. After an isotopy of $C$, one may assume that all index-0 critical points lie below the index-1 critical points. As a result, the knot $K'$ can be obtained as follows: start with the knot $K \subset S^1 \times S^2$; add $m$ small unknots $x_i$ unlinked from $K$ and each other; add $m$ bands along paths $\gamma_i$ from $K$ to $x_i$, and perform band surgery along them. These bands correspond to the index-1 critical points of the Morse function and may loop around the $S^1$-factor of $S^1\times S^2$. Note that, a priori, some of these bands may connect two unknots $x_i$ and $x_j$ instead of $K$, but we can make sure this does not happen by another isotopy of $C$. Also note that a path $\gamma_i$ might go through the band associated to $\gamma_j$ if $\gamma_i$ appears later. However, by isotoping $\gamma_i$ across $\gamma_j$, it can be made to go through $x_j$ instead. This allows us to assume that the bands are disjoint from each other. With this description of $C$ in mind, we proceed towards a proof of Theorem~\ref{thm1}.\\

Let $K, K', C, M$, and $M'$ be as in the statement of the theorem and let $k \in \BZ$ be the framing of $K$. This framing can be made canonical by using the dotted unknot diagram for $S^1 \times S^2$; moreover, since $C$ preserves the framing, $K'$ must also have framing $k$. We perform $k$-framed surgery on the concordance $C$ to obtain a ribbon homology cobordism $W_C$ from $\partial M$ to $\partial M'$ as follows. 

The neighborhood of $C \cong S^1 \times I$ can be identified as $\nu(C) \cong S^1 \times I \times D^2$. 
Let $\phi : S^1 \times S^1 \rightarrow S^1 \times S^1$ be the map associated to Dehn surgery with coefficient $k$, where the second $S^1$ factor is treated as a meridian. Multiplying $\phi$ by the identity map $\Id : I \rightarrow I$, we obtain a diffeomorphism $\Phi : S^1 \times S^1 \times I \rightarrow S^1 \times S^1 \times I$. Define $W_C$ as

$$W_C \coloneq S^1 \times S^2 \times I \backslash \nu(C) \cup_{\Phi} S^1 \times D^2 \times I.$$

Note that $W_C$ has two boundary components, which are obtained as $k$-surgery on the knots $K$ and $K'$, respectively. It follows that $W_C$ is a cobordism from $\partial M$ to $\partial M'$, where we choose the convention that a cobordism from $Y$ to $Y'$ has oriented boundary $-Y \sqcup Y'$. While $W_C$ must be a \textit{homology} cobordism for algebraic reasons, even if $C$ is not ribbon, we instead prove this by describing its Kirby diagram, which we use to show that the links $L_C$ are in $S^4$ and not an arbitrary homotopy 4-sphere.

\begin{prop}\label{propA}
Let $\Sigma \coloneq M \cup W_C \cup -M'$, then $\Sigma \cong S^4$.
\end{prop}
\begin{proof}

The projection $\rho : S^1 \times S^2 \times I \rightarrow I$ produces a relative handle decomposition of the pair $(S^1 \times S^2 \times I, S^1 \times S^2 \times \{0\})$ with no critical points. We investigate the complement $S^1 \times S^2 \times I \backslash \nu(C)$ using the techniques from \cite[Chapter~6.2]{Gompf-Stipsicz}. After isotopy, we can assume $\rho$ restricts to a Morse function on $C$ with no local maxima, and with local minima and saddles as above. This yields a relative handle decomposition of the pair $$(S^1 \times S^2 \times I \backslash \nu(C), S^1 \times S^2 \backslash \nu(K)),$$ which can be described as follows.

We draw $S^1 \times S^2 \backslash \nu(K)$ as in Figure~\ref{fig2} (center), where the dotted line represents a deleted neighborhood. Each local minimum of $\rho|_C$ gives birth to a 1-handle represented by a dotted unknot, which we denote $x_i$ by a mild abuse of notation. A saddle move along a band $\gamma_i$ yields a 0-framed 2-handle $y_i$ in the following way.
By our assumption, the band $\gamma_i$ must connect $K$ to $x_i$. Letting $\mu_K$ and $\mu_i$ be the oriented meridians of $K$ and $x_i$, the attaching circle of $y_i$ is obtained by connecting $-\mu_K$ and $\mu_i$ using the band $\gamma_i$. Note that the band $\gamma_i$ may be linked with the dotted unknot $x$ defining $S^1 \times S^2$, as well as the deleted neighborhood of $K$. See Figure~\ref{fig3} (top left) for an example.

\begin{figure}[htbp]
  \centering
  \includegraphics[width=0.25\textwidth]{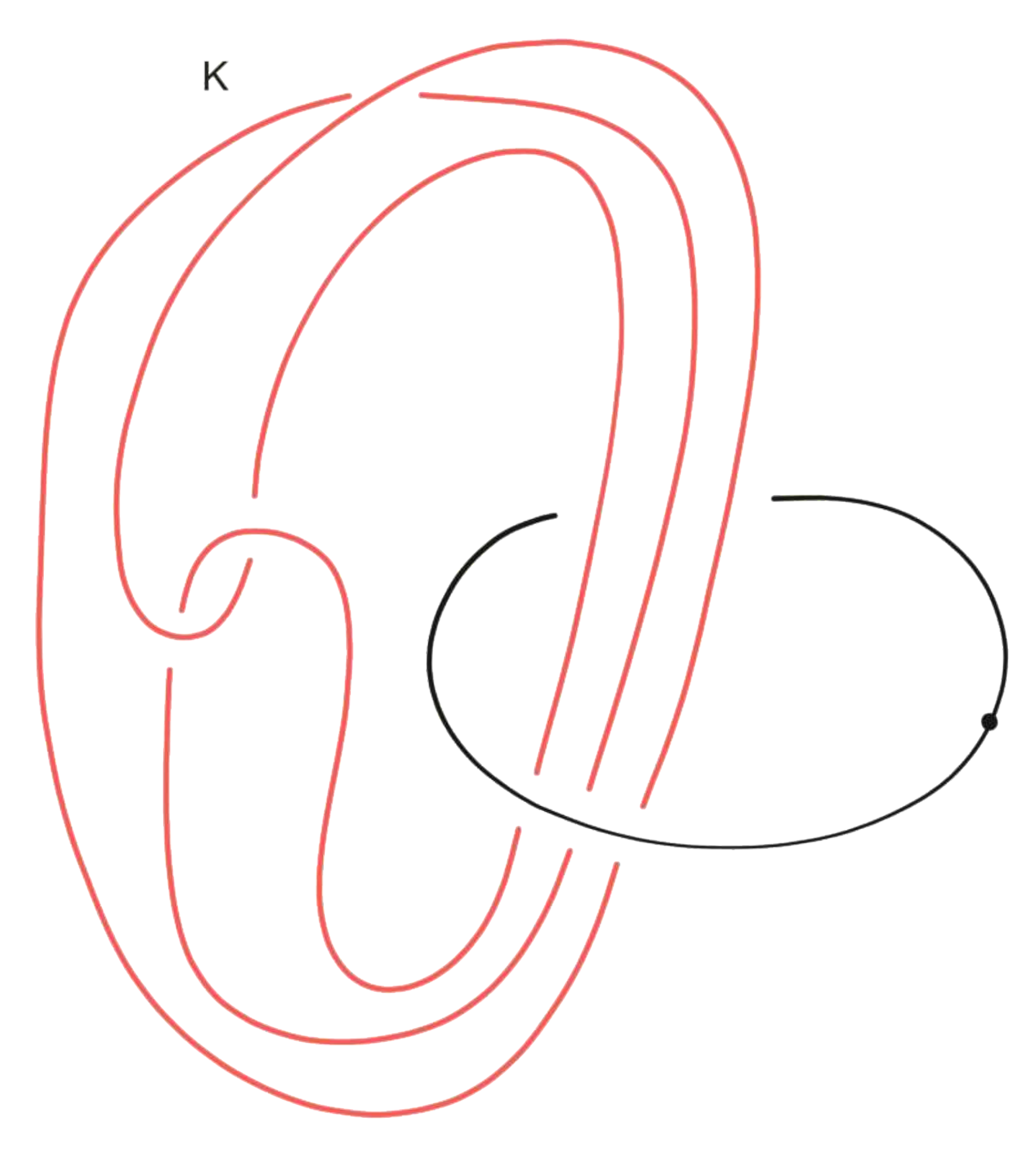} 
  \hspace{0.5 cm}
  \includegraphics[width=0.25\textwidth]{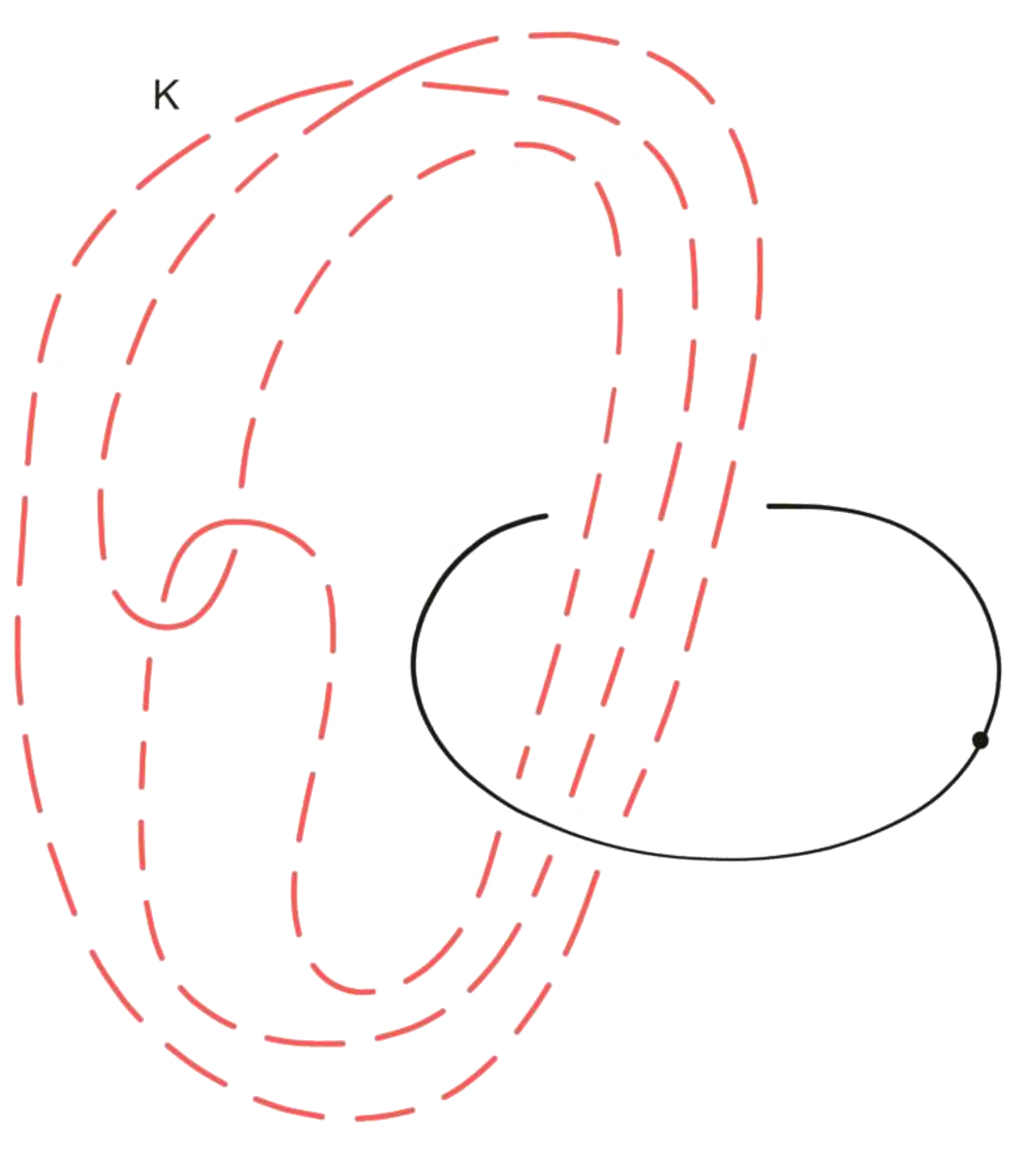}
  \hspace{0.5 cm}
  \includegraphics[width=0.27\textwidth]{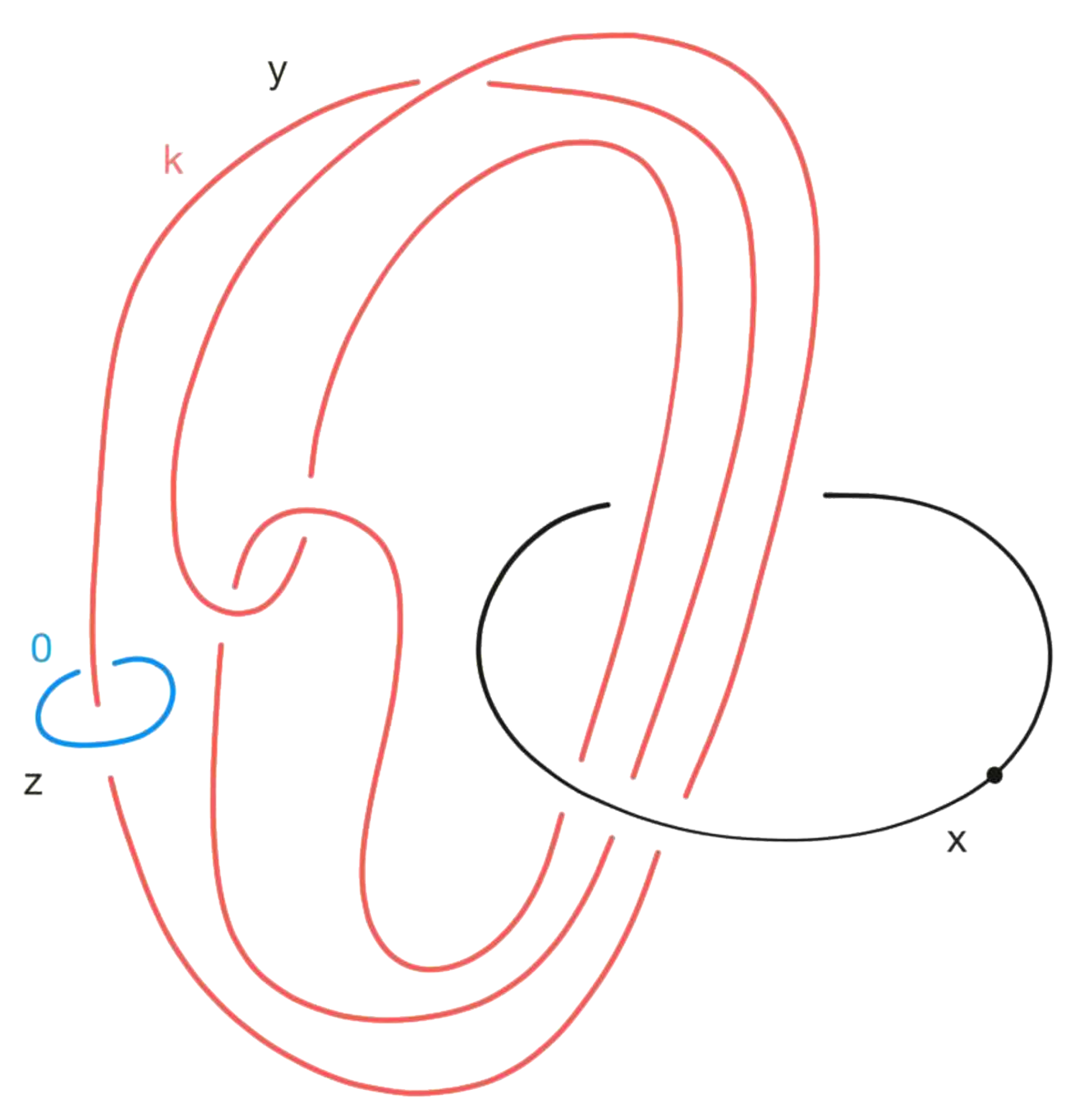}
  \caption{Left: an example of a knot $K \subset S^1 \times S^2$. Center: the complement $S^1 \times S^2 \backslash \nu(K)$. Right: a Kirby diagram of the double $M \cup -M = S^4$.}
  \label{fig2}
\end{figure}

\begin{figure}[htbp]
  \centering
  \includegraphics[width=0.42\textwidth]{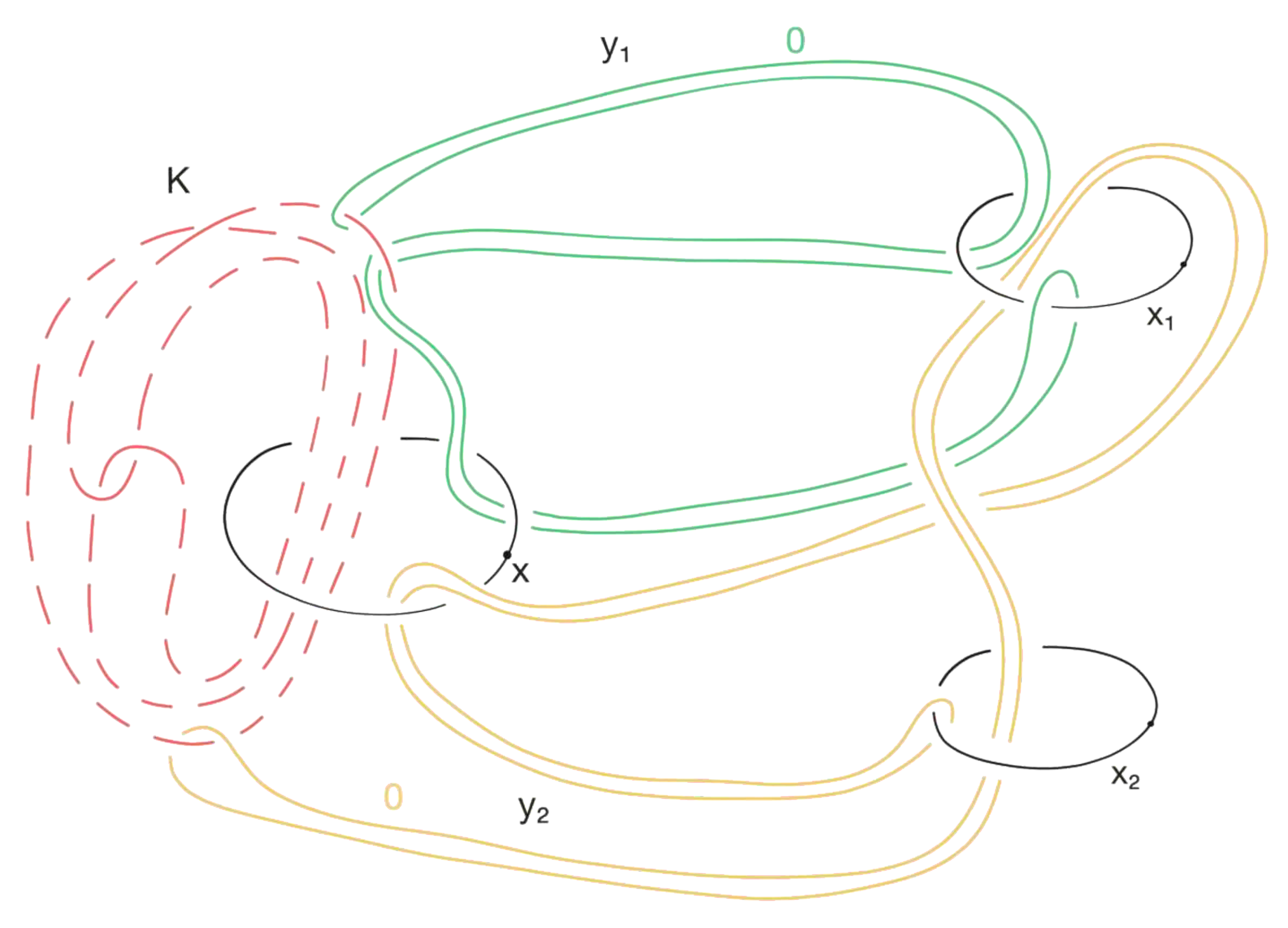} 
  \includegraphics[width=0.42\textwidth]{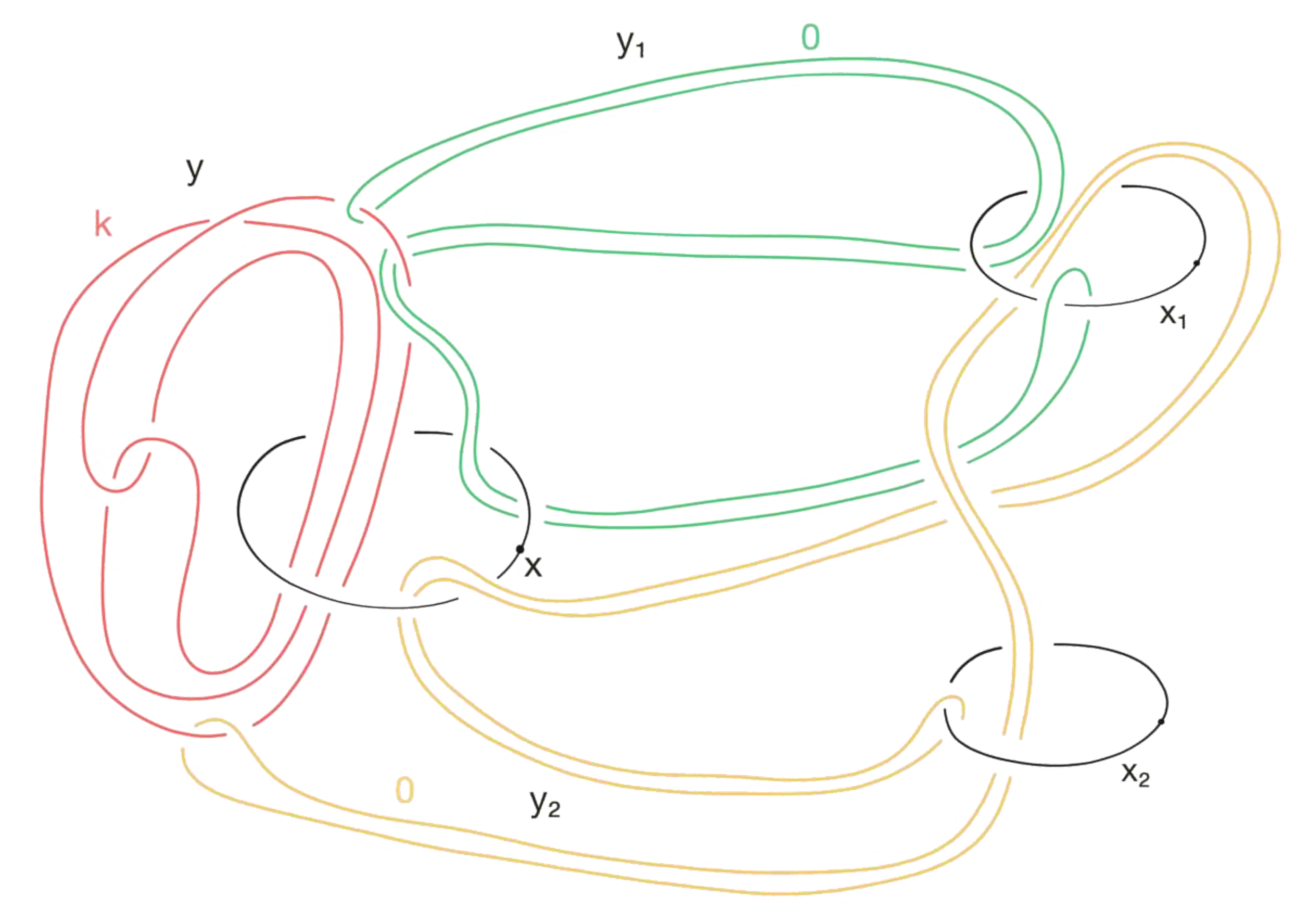}
  \includegraphics[width=0.42\textwidth]{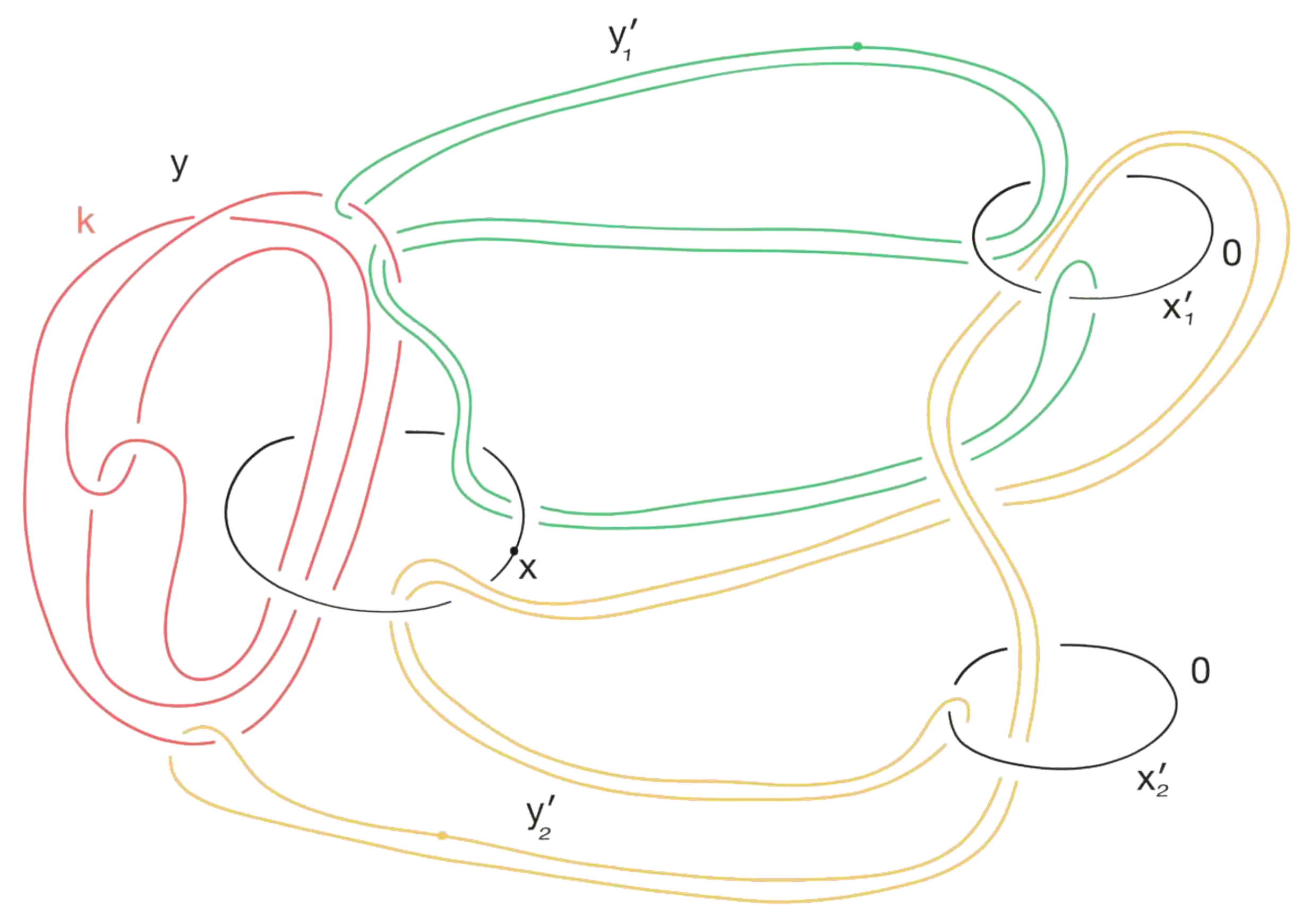} 
  \includegraphics[width=0.42\textwidth]{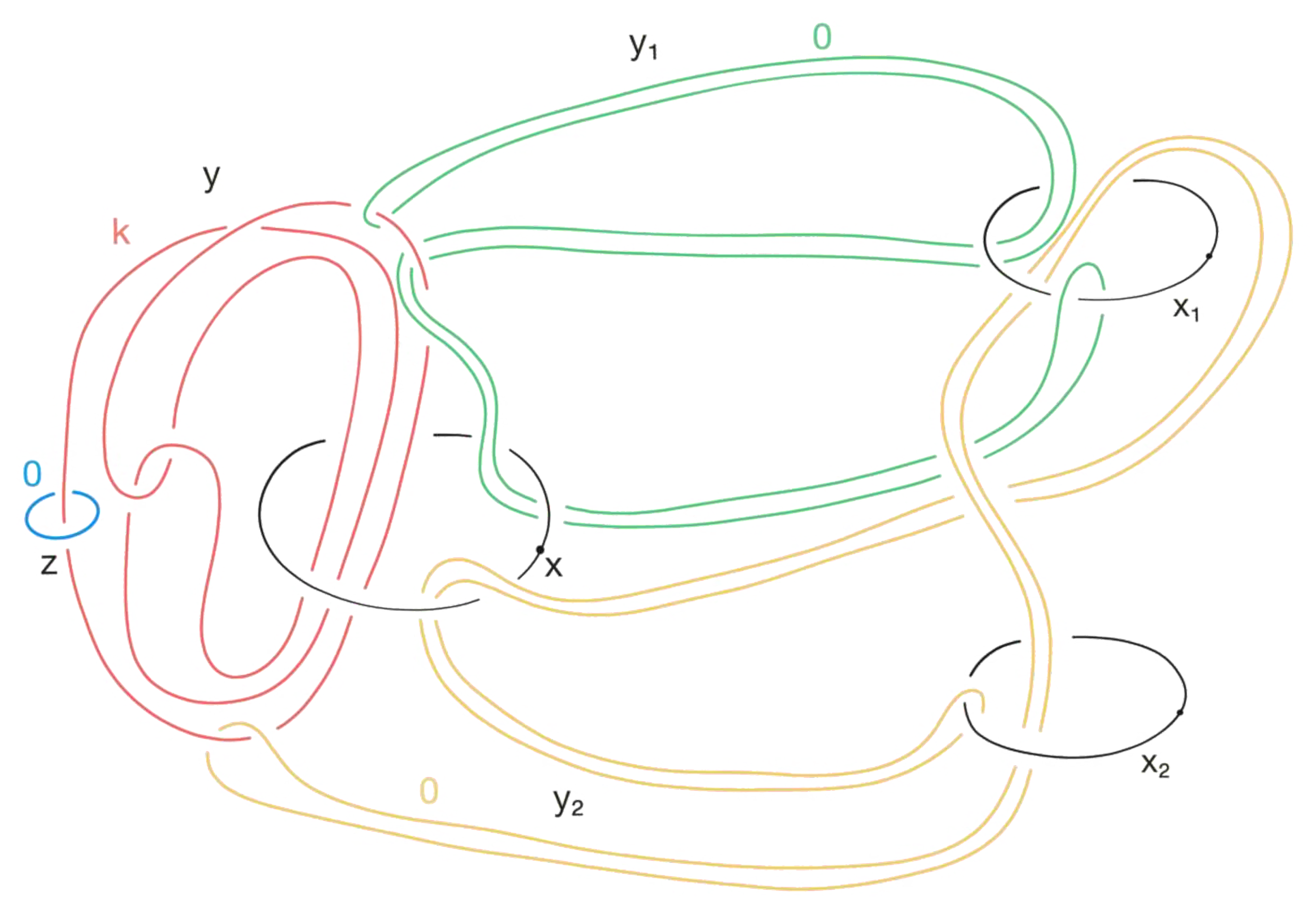}
  \caption{Top left: a relative Kirby diagram of $(S^1 \times S^2 \times I \backslash \nu(C), S^1 \times S^2 \backslash \nu(K))$. Top right: a Kirby diagram of $X = M \cup W_C$, and also of the pair $(W_C, \partial M)$ if we treat $x, y$ as a surgery diagram rather than handles. Bottom left: a Kirby diagram of $X'$, obtained from $X$ by a zero-dot exchange on the pairs $x_i, y_i$. Bottom right: a Kirby diagram of $\Sigma = M \cup W_C \cup -M'$.}
  \label{fig3}
\end{figure}

\begin{figure}[htbp]
  \centering
  \includegraphics[width=0.42\textwidth]{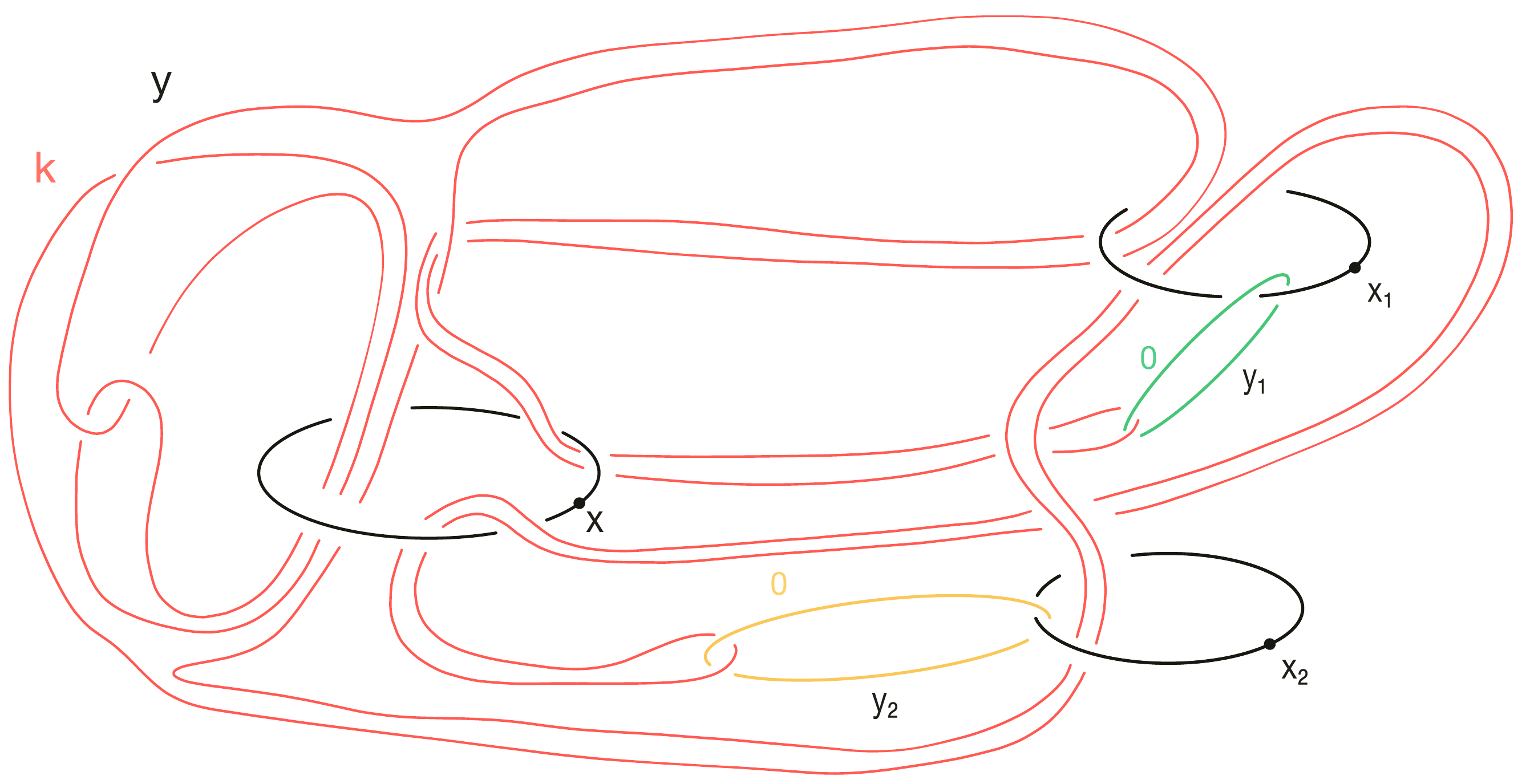} 
  \includegraphics[width=0.42\textwidth]{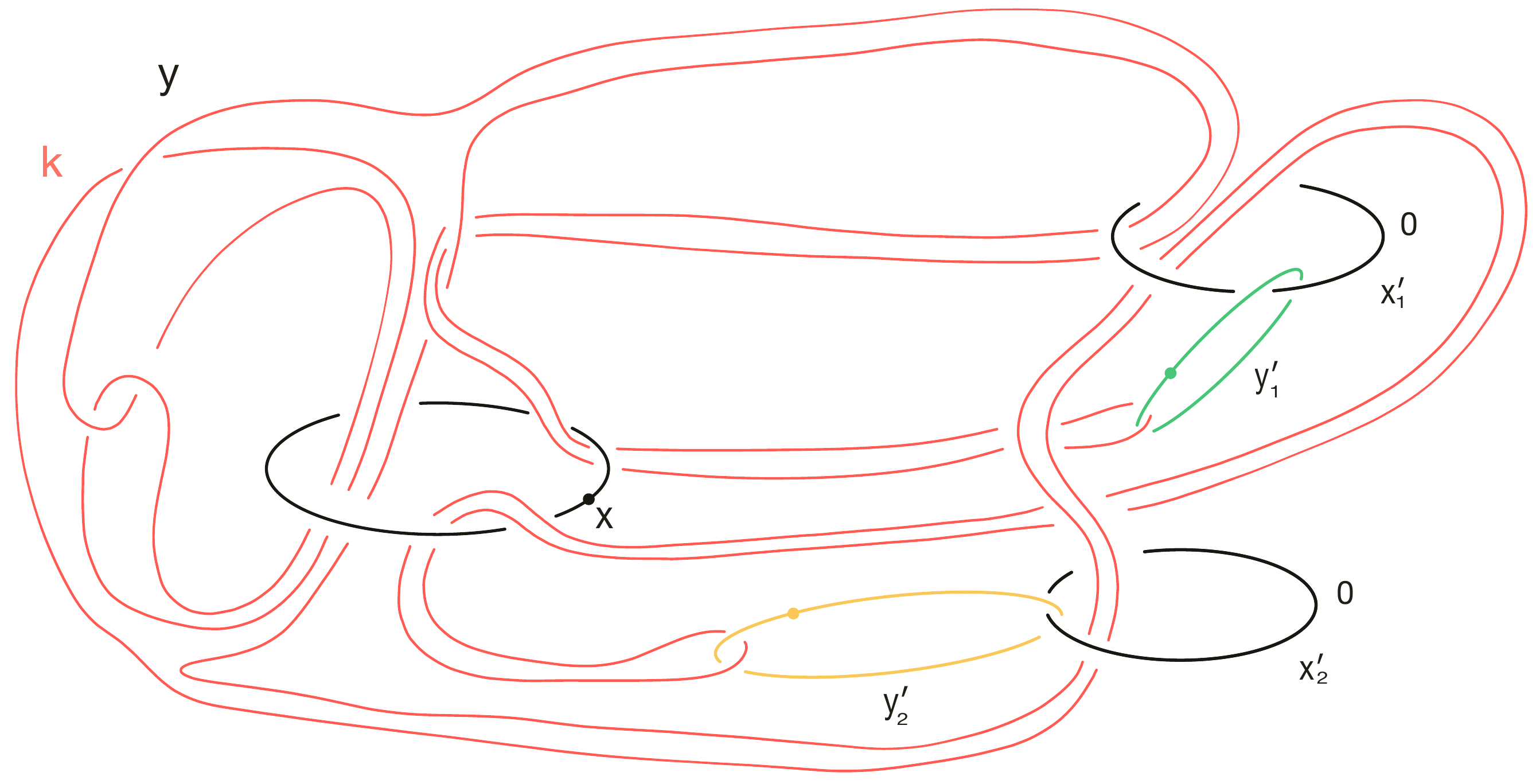}
  \caption{An illustration of the diffeomorphism $M \cup W_C \cong M'$.\\ Left: a Kirby diagram of $M \cup W_C$ obtained from Figure~\ref{fig3} (top right) by isotopy.\\ Right: a Kirby diagram of $M'$ obtained from Figure~\ref{fig3} (bottom left) by isotopy. The two are related by a zero-dot exchange on the pairs $x_i, y_i$, which, when viewed as an involution of the boundary, extends to a diffeomorphism of the 4-manifolds.}
  \label{fig4}
\end{figure}

It remains to reglue the neighborhood $\nu(C) \cong S^1 \times D^2 \times I$ using the framing coefficient $k$. This corresponds to filling in the dotted line in the picture and writing $k$ next to it. We obtain a relative handle decomposition of $(W_C, (S^1 \times S^2)_k(K))$ as in Figure~\ref{fig3} (top right).
Since our surgery diagram of $(S^1 \times S^2)_k(K)$ is naturally identified with $\partial M(K)$, this also represents an absolute handle decomposition of $X \coloneq M \cup W_C$ if we treat $x$ as a 1-handle and let $y$ be a $k$-framed 2-handle attached along $K$. 
Note that since $\partial M$ is a homology sphere, the meridian $\mu_K$ is null-homologous, which readily implies that $W_C$ is a ribbon homology cobordism because every 2-handle $y_i$ homologically cancels the 1-handle $x_i$. 

Since $\partial X \cong \partial M'$, we can glue on a copy of $-M'$ to obtain $\Sigma = X \cup -M'$.
First, we observe that $\Sigma$ is a homotopy 4-sphere. Indeed, $\pi_1(X) = 1$ since each relation $y_i$ cancels the generator $x_i$, hence $\pi_1(X \cup -M') = 1$ by van Kampen, while the homology of $\Sigma$ is computed by a simple Mayer--Vietoris argument.

We claim that $\Sigma \cong S^4$, which provides a link
$$L_C : M \sqcup -M' \hookrightarrow \Sigma = S^4.$$
To show this, we must first understand the gluing map $\phi_C : \partial X \rightarrow \partial M'$.

Let $X'$ be given by the Kirby diagram of Figure~\ref{fig3} (bottom left), which is obtained from the diagram of $X$ (top right) by performing a zero-dot exchange on the pairs $x_i$ and $y_i$, resulting in 1-handles $y_i'$ and 2-handles $x_i'$. This is made possible by the observation that the $y_i$ form an unlink since each $y_i$ bounds a disc $D_i$ obtained by connecting the meridional discs of $K$ and $x_i$ by the band $\gamma_i$. 
The zero-dot exchange provides a boundary diffeomorphism $\partial M' = \partial X \xrightarrow{\cong} \partial X'$; we claim that $X' \cong M'$.

Note that while $y_i$ may intersect the disc bounded by $x_i$ multiple times, $x_i'$ intersects the disc $D_i$
bounded by $y_i'$ only once, which allows us to cancel the handles $y_i'$ and $x_i'$. However, before doing so, we must slide every other 2-handle that goes through $y_i'$ over $x_i'$, until no other 2-handle intersects the disc $D_i$. The only other 2-handle intersecting $D_i$ is $y$ (see picture); thus we must slide $y$ over $x_i'$. This slide is given precisely by the band $\gamma_i$, which means that after we complete the slides and cancellations, we are left with the 1-handle $x$ and a 2-handle $y'$ attached along the knot $K' \subset S^1 \times S^2$, as in Figure~\ref{fig5}.
Since this is the Kirby diagram of $M'$, we conclude that $X' \cong M'$.

\begin{figure}[htbp]
  \centering
  \includegraphics[width=0.56\textwidth]{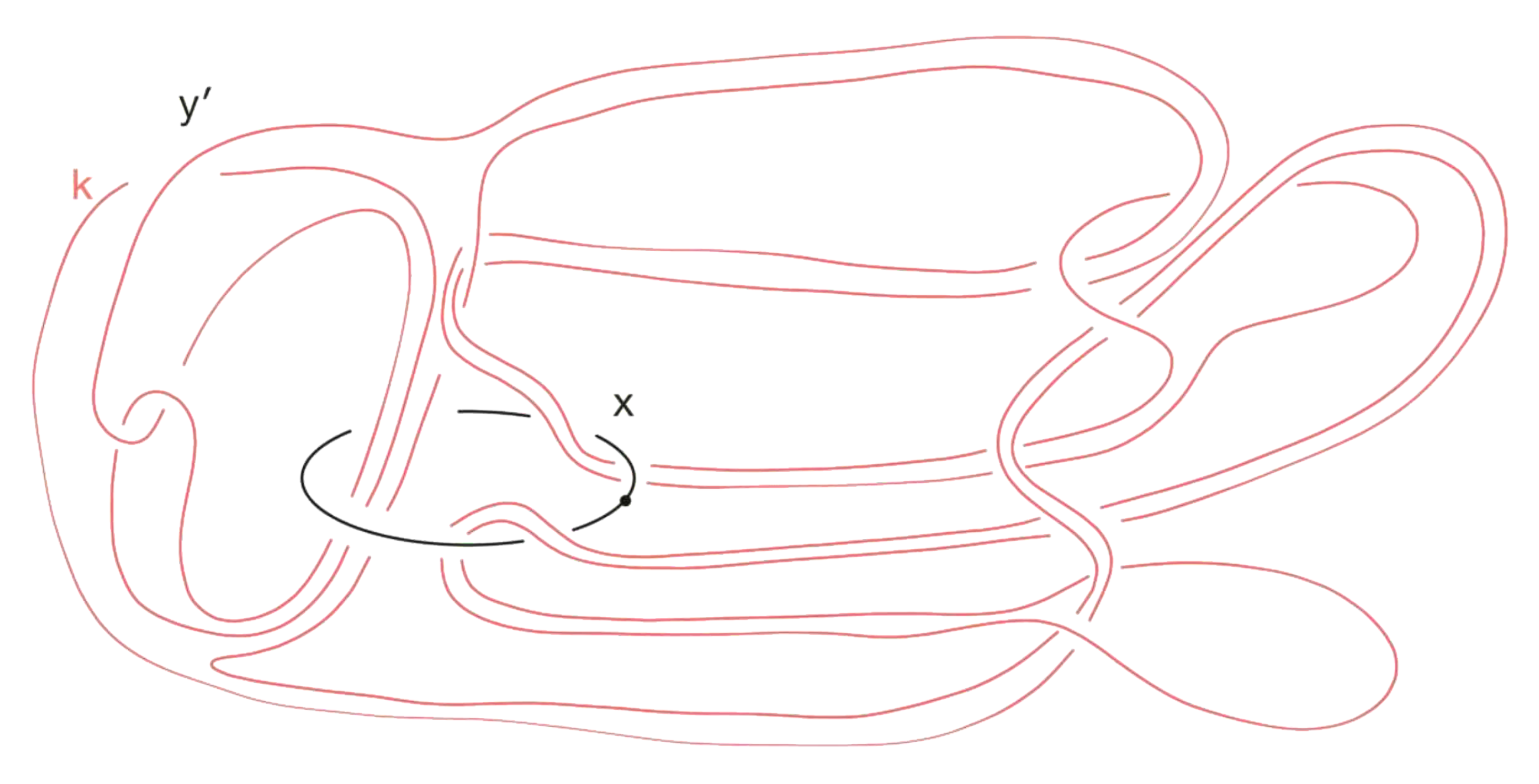}
  \caption{A Kirby diagram exhibiting $X'$ as the Mazur manifold $M'$.}
  \label{fig5}
\end{figure}

Finally, observe that in order to obtain the above diffeomorphism, we slid $y$ over other 2-handles, but never slid a 2-handle over $y$. This implies that the meridian $\mu_K$ of $y$ coincides with the meridian $\mu_{K'}$ of $y'$, or, in other words, $\mu_K$ is fixed by the boundary diffeomorphism $\partial X \rightarrow \partial M'$. It follows that a Kirby diagram of $\Sigma = X \cup -M'$ is obtained from the diagram of $X$ by simply attaching a 0-framed 2-handle $z$ along $\mu_K$; see Figure~\ref{fig3} (bottom right). We slide every 2-handle $y_i$ over $z$ so as to unlink it from $y$ at the meridian $\mu_K$. As a result, we can isotope $y_i$ across the band $\gamma_i$, which allows us to cancel $y_i$ with $x_i$. This leaves us with the handles $x$, $y$, and $z$, which form the Kirby diagram of the double of $M$, i.e., $S^4$ (see Figure~\ref{fig2} (right)). This concludes the proof and provides the link $L_C$. 
\end{proof}

\begin{prop}\label{propB}
The restrictions $L_C|_M$ and $L_C|_{M'}$ are smoothly standard.   
\end{prop}
\begin{proof}
According to Proposition~\ref{propEx}, it suffices to show that there are diffeomorphisms $-M' \cup W_C \cong -M$ and $M \cup W_C \cong M'$ extending the boundary identifications. 

For the former, simply note that we showed, in the proof of Proposition~\ref{propA}, that after handle slides and cancellations $M \cup W_C \cup -M'$ admits a Kirby diagram of the double of $M$ as in Figure~\ref{fig2} (right). Moreover, since the handle slides did not involve the handles $x,y$, it follows that $-M' \cup W_C \cong -M$. Alternatively, one can start with the Kirby diagram of $M'$ as in Figure~\ref{fig3} (bottom left) and glue on $-W_C$ upside down, which has the effect of adding $0$-framed meridians to the $y_i'$. Then after handle slides and cancellations, we are left with the handles $x$ and $y$, which form $M$. 

The proof of the second statement is only slightly more involved. The 4-manifolds $M \cup W_C$ and $M'$ admit Kirby diagrams as in Figure~\ref{fig3} (top right) and (bottom left), respectively. These are related by a zero-dot exchange. We claim that this involution on the boundary extends to a diffeomorphism $M \cup W_C \cong M'$ similarly to \cite[Proposition~5.1]{Yasui}. The reason this works is because $y_i$ can be unlinked from $x_i$ by a single crossing change at the meridian of $x_i$. 

We begin with the diagram of Figure~\ref{fig3} (top right) and transform it into the diagram of Figure~\ref{fig3} (bottom left). First, we isotope the 2-handles $y_i$ across the bands $\gamma_i$, while carrying a portion of $y$ along, obtaining Figure~\ref{fig4} (left). Note that each pair $(x_i, y_i)$ forms a Hopf link and does not interact with other 1-handles. This means that these pairs describe embedded 4-balls, and performing a zero-dot exchange on them corresponds to regluing these 4-balls by some involution of the boundary. Since every diffeomorphism of $S^3$ extends to $B^4$, this preserves the diffeomorphism type of the ambient 4-manifold. After the zero-dot exchange, we obtain Figure~\ref{fig4} (right), where we relabel the handles by adding a prime symbol. This diagram differs from Figure~\ref{fig3} (bottom left) by isotopy, and thus represents $M'$.

Lastly, note that the two diffeomorphisms we described extend the natural boundary maps $\partial (-M' \cup W_C) \rightarrow \partial (-M)$ and $\partial (M \cup W_C) \rightarrow \partial M'$. This implies that the restrictions of $L_C$ to $M$ and $M'$ are smoothly standard. 
\end{proof}

\begin{proof}[Proof of Theorem~\ref{thm1}]
It remains to obstruct a splitting of $L_C$ by a topological $S^3$.


So far we have seen that the individual components of the link $L_C$ are smoothly, and hence also topologically, unknotted. Together with the generalized Schoenflies theorem \cite{Schoenflies}, this implies that $L_C$ is topologically isotopic to the unlink $L_0 : M \sqcup -M' \hookrightarrow S^4$ (defined in Section~\ref{2.2}) if and only if $L_C$ is split by a topological $S^3$. We have seen earlier that the complement $W_0 = S^4 \backslash L_0$ is simply connected. On the other hand, the fundamental group of the complement $W_C =  S^4 \backslash L_C$ is easily seen to be nontrivial. Indeed, since $W_C$ is a ribbon homology cobordism, it follows from \cite[Theorem~1.14]{Metatheorem} that the map $\pi_1(\partial M) \rightarrow \pi_1(W_C)$ induced by inclusion is injective. Since we're assuming that $M$ is not the 4-ball, we have $\pi_1(\partial M) \neq 1$ (since $B^4$ is the only Mazur manifold bounded by $S^3$, see \cite{Conway-Tosun}), therefore $\pi_1(W_C) \neq 1$. It follows that $L_C$ is not topologically isotopic to $L_0$, and thus it cannot be split by a topological $S^3$.      
\end{proof}

We observe that since $W_C$ is a ribbon homology cobordism from $\partial M$ to $\partial M'$, the other results of \cite{Metatheorem} apply. In particular, their Theorem 1.19 implies that the cobordism maps $F_{W_C}^{\circ}$ include $HF^{\circ}(\partial M)$ into $HF^{\circ}(\partial M')$ as a direct summand, where $\circ \in \{+, -, \infty, \widehat{\phantom{x}}\}$. Here, $HF$ denotes Heegaard Floer homology with its usual flavors, and coefficients $\BZ_2$, as introduced by Ozsváth and Szabó in \cite{OS3}, while $F_{W_C}^{\circ}$ is the cobordism map defined in \cite{OS2}, where we sum over all spin$^c$ structures. If the cobordism $W_C$ is smoothly split by some homology 3-sphere $Z$, then the inclusion $F_{W_C}^{\circ}$ must factor through $HF^{\circ}(Z)$, which yields an inclusion $HF^{\circ}(\partial M) \hookrightarrow HF^{\circ}(Z)$. This implies that any homology sphere $Z$ smoothly splitting $L_C$ cannot be \enquote{smaller} than $\partial M$. In particular, it follows from Lemma~\ref{lem2} below that $L_C$ is not smoothly split by $S^3$, independently of the topological obstruction.

\begin{remark}\label{A2}
It turns out that if we relax the Mazur condition on one of the two manifolds, allowing it to be chosen from the second filtration level $\A^2$, then there is a much simpler construction. For arbitrary Mazur manifolds $M_1$ and $M_2$, the decomposition $S^4 = M_1 \natural M_2 \cup - (M_1 \natural M_2)$ determines a link
$$L : M_1 \sqcup -(M_1 \natural M_2) \hookrightarrow S^4,$$
whose complement is a ribbon homology cobordism consisting of the 1- and 2-handles from the Kirby diagram of $M_2$. Then the statement of Theorem~\ref{thm1} follows for $L$ analogously.
\end{remark}

\section{2-component links in $\#^n \BC \BP^2$}\label{4}
We have seen that the links $L_C$ are not split even by a topological $S^3$, while the individual components are topologically unknotted. As part of our proof of Theorem~\ref{thm2}, we describe a general strategy for turning such links $L_C \subset S^4$ into topologically standard links $L_{C, n} \subset \#^n \BC \BP^2$. An important feature of this procedure is that the smooth obstruction survives this \enquote{blow-up}.

Fix $M, M', C$, and let $L_C : M \sqcup -M' \hookrightarrow S^4$ be the link from Theorem~\ref{thm1}. Each component of $L_C$ is topologically standard, while the complement $ W_C = S^4 \backslash L_C$ has $\pi_1(W_C) \neq 1$. Our goal is to add 2-handles to $W_C$ so as to kill this fundamental group.

In fact, this can be accomplished by attaching a single 2-handle to $\partial M$ along the meridian $\mu_K$ of the 2-handle of $M$. To see this, recall that any Mazur manifold $M$ doubles to $S^4$: $M \cup -M = S^4$. A Kirby diagram for $M \cup -M$ is obtained from that of $M$ by adding a 0-framed 2-handle along $\mu_K$. Similarly, one obtains a relative Kirby diagram for $(-M, \partial M)$ by attaching a 0-framed 2-handle to $\partial M \times I$ along $\mu_K$, where we use the surgery diagram of $\partial M$ coming from $M$ and omit the 3- and 4-handles of $-M$. Since $-M$ is simply connected, this handle attachment must kill $\pi_1(\partial M \times I) = \pi_1(\partial M)$. Since the meridian $\mu_K$ is now null-homotopic, it can be checked that each relation coming from a 2-handle $y_i$ simply cancels the generator corresponding to the 1-handle $x_i$ in the presentation of the fundamental group of $W_C$ with an extra relation coming from the 2-handle at $\mu_K$. 

For now, assume $n < 0$, and choose a Kirby diagram of $\#^n \BC \BP^2$ consisting of $|n|$ unknots with framing $-1$. Adding these 2-handles to our picture of $L_C \subset S^4$, we obtain a \enquote{blown-up} link $$L_B : M \sqcup -M' \hookrightarrow \#^n \BC \BP^2;$$ see Figure~\ref{fig6} (left).

Next, we slide the 2-handle of $M \subset \#^n \BC \BP^2$ once over each of the $(-1)$-framed 2-handles. As a result, its framing increases by $n$, while the attaching circle becomes negatively linked with each of the $(-1)$-framed 2-handles; see Figure~\ref{fig6} (right). By choosing the same 0- and 1-handles from $M$, together with the new 2-handle, we obtain an embedding $M_n \hookrightarrow \#^n \BC \BP^2$. Since the handle slides above do not affect the handles of $-M' \subset \#^n \BC \BP^2$, this produces a link $$L_{C, n} : M_n \sqcup -M' \hookrightarrow \#^n \BC \BP^2.$$ 

We claim that $W_{C, n} = \#^n \BC \BP^2 \backslash L_{C, n}$ is simply connected. Indeed, the cobordism $W_{C, n}$ contains a 2-handle attached along the meridian $z_n$ of the 2-handle of $M_n$, which kills $\pi_1(\partial M_n)$. Since the curve $z_n$ is now null-homotopic, each 2-handle $y_i$ simply kills the generator corresponding to $x_i$. Moreover, since the restrictions $L_C|_M, L_C|_{M'}$ were smoothly standard, it easily follows that so are the restrictions $L_{C, n}|_{M_n}, L_{C, n}|_{M'}$. For the latter, simply note that the blow-up and handle slides occur away from $M'$. For the former, observe that the complement $\#^n \BC \BP^2 \backslash L_{C, n}(M_n)$ is obtained from $S^4 \backslash L_C(M) = -M$ by attaching $|n|$ $-1$-framed 2-handles along the meridian of the 2-handle of $-M$. Sliding this 2-handle under the meridional 2-handles, we obtain a Kirby diagram of $-M_n \# \BC \BP^2$, which is exactly the complement of the standard embedding $M_n \hookrightarrow \#^n \BC \BP^2$. Since this diffeomorphism $\#^n \BC \BP^2 \backslash L_{C, n}(M_n) \cong -M_n \#^n \BC \BP^2$ clearly restricts to identity on $\partial M_n$, the restriction $L_{C, n}|_{M_n}$ must be smoothly standard by Proposition~\ref{propbigEx}.

\begin{figure}[htbp]
  \centering
  \includegraphics[width=0.42\textwidth]{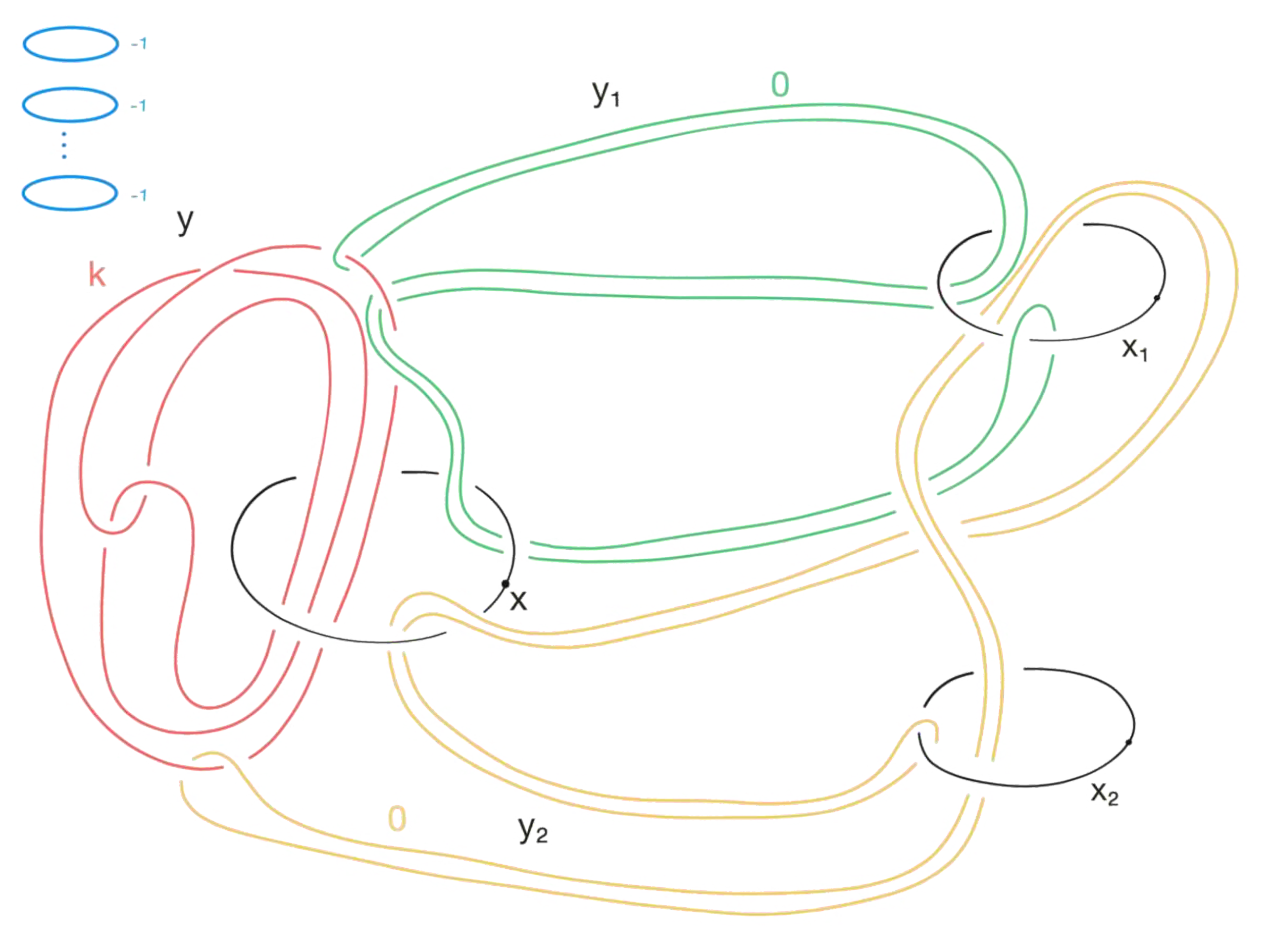} 
  \includegraphics[width=0.42\textwidth]{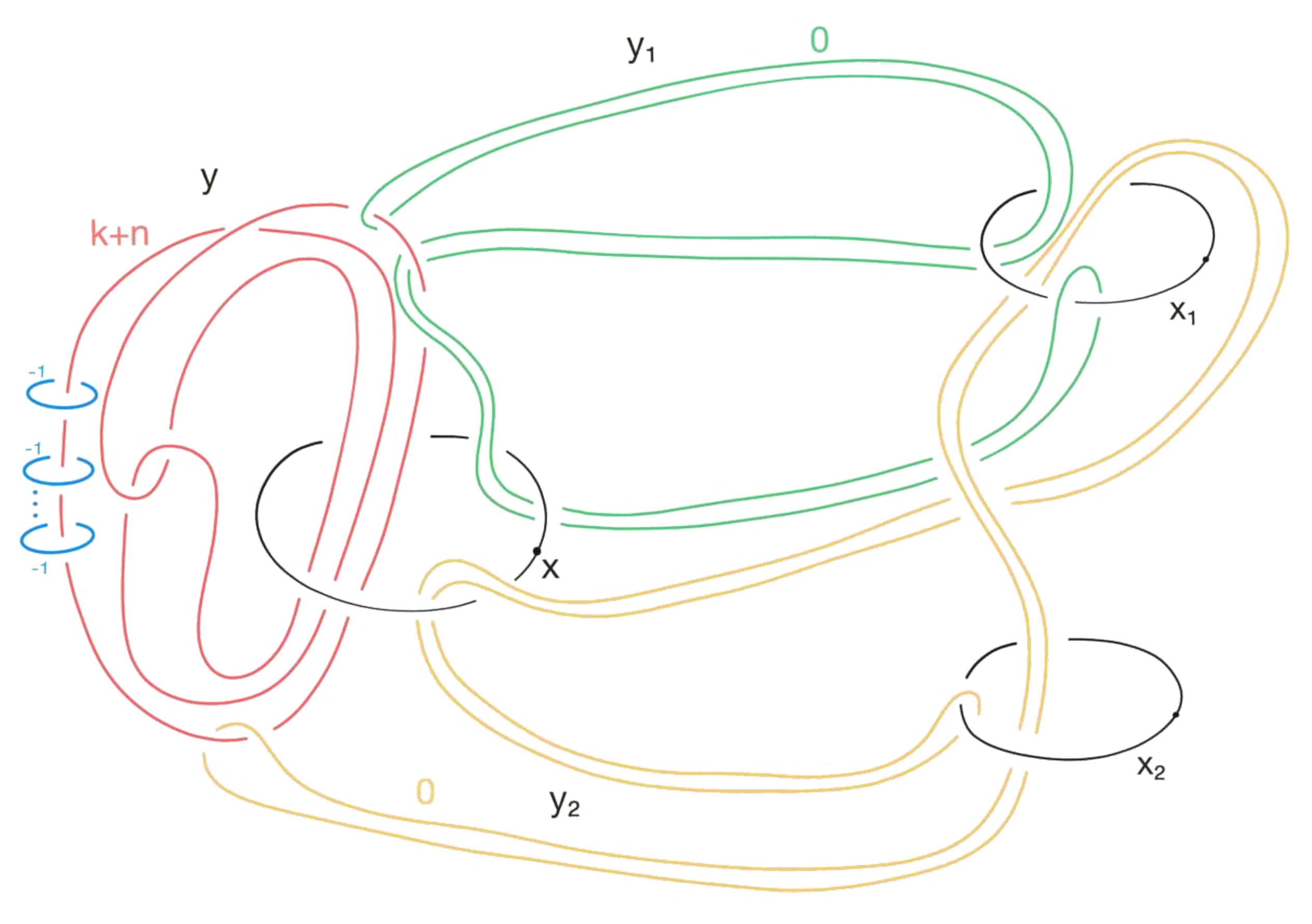}
  \caption{Left: a Kirby diagram of $M \cup W_C\#^n \BC \BP^2$. Right: the diagram obtained from the first one by sliding $y$ over the blue handles; this describes the links $L_{C, n}$.}
  \label{fig6}
\end{figure}

\begin{proof}[Proof of Theorem~\ref{thm2}]

We compare the link $L_{C, n}$ to the unlink $L_0 : M_n \sqcup -M' \hookrightarrow \#^n \BC \BP^2$, which was defined in Section~\ref{2.4}. Since $\#^n \BC \BP^2$ is smooth, closed, and simply connected, while $W_{C, n}$ is simply connected, it follows from Proposition~\ref{biglink} that $L_{C, n}$ is topologically isotopic to $L_0$. In particular, $L_{C, n}$ must be split by a topological $S^3$. 

In order to prove the smooth part of the theorem, we investigate $W_{C, n}$ as a smooth cobordism from $\partial M_n$ to $\partial M'$. Our goal is to show that the associated cobordism map $$F_{W_{C, n}} : HF^+_{\mathrm{red}}(\partial M_n) \rightarrow HF^+_{\mathrm{red}}(\partial M')$$ is an inclusion, where $HF^{+}_{\mathrm{red}}$ is the reduced part of $HF^{+}$ (see \cite{OS1}) and we are summing over all spin$^c$ structures. In that case, if $L_{C, n}$ were smoothly split by $S^3$, the map $F_{W_{C, n}}$ would factor through $HF^+_{\mathrm{red}}(S^3) = 0$, which would imply $HF^+_{\mathrm{red}}(\partial M_n) = 0$, i.e., $\partial M_n$ is an $L$-space. Then Lemma~\ref{lem2} below provides the desired contradiction.

We decompose $W_{C, n}$ as $W_{C,n} = W_n \cup W_C$, where $W_n$ is the cobordism from $\partial M_n$ to $\partial M$ consisting of the $|n|$ $(-1)$-framed 2-handles attached along $z_n$, while $W_C$ is the cobordism from $\partial M$ to $\partial M'$ coming from the link $L_C \subset S^4$. This yields a decomposition of $F_{W_{C, n}}$ as
$$F_{W_{C, n}} : HF^+_{\mathrm{red}}(\partial M_n) \xrightarrow{F_{W_n}} HF^+_{\mathrm{red}}(\partial M) \xrightarrow{F_{W_C}} HF^+_{\mathrm{red}}(\partial M').$$
We have seen earlier that $W_C$ is a ribbon homology cobordism and its induced map on $HF^{+}$ is an inclusion. Since $U$-torsion elements must be mapped to $U$-torsion elements, the map $F_{W_C}$ on $HF^+_{\mathrm{red}}$ must also be an inclusion. It would then suffice to show that $F_{W_n}$ is an isomorphism. 

We decompose the cobordism $W_n$ into elementary cobordisms with $b_2 = 1$: $$W_n = C_{1} \circ \dots  \circ C_{|n+1|}  \circ C_{|n|}.$$ Here, $C_{|n|}$ is obtained from $\partial M_n \times I$ by attaching a $(-1)$-framed 2-handle along $z_n$. It can be easily checked that $\partial_+C_{|n|} = \partial M_{n+1}$. Similarly, $C_{|n+1|}$ is the cobordism obtained from $\partial M_{n+1} \times I$ by attaching a $(-1)$-framed 2-handle along the meridian $z_{n+1}$ of the 2-handle of $M_{n+1}$. It follows by induction that $C_{|k|}$ is an elementary cobordism from $\partial M_{k}$ to $\partial M_{k+1}$. Letting $F_{C_{|k|}} : HF^+_{\mathrm{red}}(\partial M_k) \rightarrow HF^+_{\mathrm{red}}(\partial M_{k+1})$ be the associated cobordism map, $F_{W_n}$ decomposes as $$F_{W_n} = F_{C_{1}} \circ \dots \circ  F_{C_{|n+1|}} \circ F_{C_{|n|}}.$$ Then the proof of the case $n < 0$ follows from Lemma~\ref{lem1} below.

Finally, we recover the case $n > 0$. Let $-C$ be the concordance from $-K$ to $-K'$ obtained by reversing the orientation of $C$. Theorem~\ref{thm1} provides a link $L_{-C} : -M \sqcup M' \hookrightarrow S^4$. Applying the blow-up construction with $-n$ produces links $L_{-C, -n} : (-M)_{-n} \sqcup M' \hookrightarrow \#^n \overline{\BC \BP^2}$, which are split topologically but not smoothly. Since $-((-M)_{-n}) = M_n$, we can simply reverse the orientation to obtain links 
$$L_{C, n} : M_n \sqcup -M' \hookrightarrow \#^n \BC \BP^2.$$
\end{proof}

\begin{lem}\label{lem1}
Every cobordism map $F_{C_i}$ above is an isomorphism.    
\end{lem}
\begin{proof}
This is already demonstrated in \cite{Akbulut-Karakurt}, as part of the proof of Proposition 1.2. Using the surgery exact triangle for Heegaard Floer homology \cite{OS1}, they show that the elementary cobordisms $C_{|k|}$ from $\partial M_k$ to $\partial M_{k+1}$ --- consisting of a $(-1)$-framed 2-handle attached along the meridian of the 2-handle of $M_k$ --- induce isomorphisms $F_{C_{|k|}} : HF^+_{\mathrm{red}}(\partial M_k) \xrightarrow{\cong} HF^+_{\mathrm{red}}(\partial M_{k+1})$.
\end{proof}

\begin{lem}\label{lem2}
For any Mazur manifold $M$, its boundary $Y = \partial M$ is not an $L$-space unless $M \cong B^4$.
\end{lem}

\begin{proof}
It is shown in \cite{Conway-Tosun} that if $M$ is a Mazur manifold such that its boundary $Y$ is an irreducible $L$-space, then $Y \cong S^3$ and $M \cong B^4$. Thus, it suffices to show that $Y$ must be irreducible.

Note that for $M = M(K)$, where $K \subset S^1 \times S^2$ is a framed knot homotopic to the core, the boundary $\partial M$ is obtained by performing a Dehn surgery on $K$ with the given framing. Since $K$ is not contained in a 3-ball, its complement $N_K = S^1 \times S^2 \backslash \nu(K)$ must be irreducible. Furthermore, $N_K$ is also $\partial$-irreducible as long as $K$ is not the core circle $S^1 \times \{\mathrm{pt}\}$. Since that would imply $M \cong B^4$, we ignore that case and assume that $N_K$ is irreducible and $\partial$-irreducible. Then the main theorem of \cite{Scharlemann} applies, implying that either $\partial M$ is irreducible or $K$ is a cable knot \cite{Rolfsen}. However, a cable knot in $S^1 \times S^2$ cannot produce a homology sphere. We conclude that $\partial M$ is irreducible, and hence not an $L$-space.

Alternatively, one could simply appeal to \cite[Proposition~3.2]{Hayden-Mark-Piccirillo}.
\end{proof}

Corollaries~\ref{cor1} and \ref{cor2} follow right away. We describe relative Kirby diagrams of the exotic pairs of the former. 

The diagram of the pair $(W_{C, n}, \partial M_n)$ is obtained from Figure~\ref{fig6} (right), which represents $M_n \cup W_{C, n}$, if we instead treat $x$ and $y$ as a surgery description of $\partial M_n$, rather than 4-dimensional handles. 

The diagram of $(W_0, \partial M_n)$ is of the following form. Embedding $M_n$ standardly in a 4-ball, we get a relative diagram of $(B^4, \partial M_n)$ from Figure~\ref{fig2} (right) by replacing the framing $k$ by $k+n$ and treating $x$ and $y$ as a surgery description for $\partial M_n$ as before, where a 3-handle is omitted. Next, we add $|n|$ $(\pm 1)$-framed 2-handles away from the picture, obtaining $(\#^n \BC \BP^2\backslash L_0(M_n), \partial M_n)$. Finally, we delete a standard copy of $-M'$ diagrammatically as follows. The handle decomposition of the double $M' \cup -M' = S^4$ allows us to replace a 4-handle by a 1-handle, a $k$-framed 2-handle at $K'$, a 0-framed 2-handle at its meridian, a 3-handle, and a 4-handle. Deleting the last three handles --- which form $-M'$ --- we are left with the 1- and 2-handles of $M'$, which are attached away from the picture. This completes the relative Kirby diagram of $(W_0, \partial M_n)$. 

\begin{remark}
Using the strategy of Theorem~\ref{thm2}, one can turn the links $L : M_1 \sqcup -(M_1 \natural M_2) \hookrightarrow S^4$ from Remark~\ref{A2} into links $L_n : (M_1)_n \sqcup -(M_1 \natural M_2) \hookrightarrow \#^n \BC \BP^2$ that are split topologically but not smoothly. Comparing these to the unlinks $L_0 : (M_1)_n \sqcup -(M_1 \natural M_2) \hookrightarrow \#^n \BC \BP^2$, we see that their complements must form exotic pairs, just like in Corollary~\ref{cor1}. One can obtain their relative Kirby diagrams in a similar fashion.
\end{remark}

\section{Exotica and the zero-dot exchange}\label{5}
Recall from Subsection~\ref{1.2} that $K_0 \subset S^1 \times S^2$ is a knot with winding number $1$ and framing zero, which appears unknotted in the standard surgery diagram of $S^1 \times S^2$. This implies that $M_0 = M(K_0)$ lies in $\M_0$, but we further assume that $-M_0 \in \M_S \subset \M_0$, i.e., $\on{tb}(-M_0) > 0$. In particular, it follows that the boundary map $\theta : \partial M_0 \rightarrow \partial M^{\bullet}$ does not extend to a diffeomorphism $\Theta : M_0 \xrightarrow{\cong} M^{\bullet}$, where $M^{\bullet} = M(K^{\bullet})$ is the Mazur manifold obtained from $M_0$ by a zero-dot exchange.

Choose a ribbon concordance $C_0 \subset S^1 \times S^2 \times I$ from $K_0$ to a knot $K'$, set $M' = M(K')$, and let $$L_{C_0} : M_0 \sqcup -M' \hookrightarrow S^4$$ be the link provided by Theorem~\ref{thm1}. Replacing $M_0$ by $M^{\bullet}$ via $\theta$, we obtain a link 
$$L_{C_0}^{\bullet} : M^{\bullet} \sqcup -M' \hookrightarrow S^4,$$ 
where the ambient 4-manifold remains $S^4$ because $M^{\bullet} \cup_{\theta} -M_0 = S^4$. Finally, let $$L_{C_n}^{\bullet} : M_n^{\bullet} \sqcup -M' \hookrightarrow \#^n \BC \BP^2$$ be the link obtained from $L_{C_0}^{\bullet}$ by following the strategy preceding the proof of Theorem~\ref{thm2}. Concretely, we first \enquote{blow up} the induced decomposition of $L_{C_0}^{\bullet}$ as follows:
$$S^4 = M^{\bullet} \cup_{\theta} W_{C_0} \cup -M' \rightsquigarrow \#^n \BC \BP^2 = M^{\bullet} \#^n \BC \BP^2 \cup_{\theta} W_{C_0} \cup -M'.$$
Next, we slide the 2-handle of $M^{\bullet}$ once over each of the $(-1)$-framed 2-handles of $\#^n \BC \BP^2$, obtaining
$$M^{\bullet}_n \hookrightarrow M^{\bullet} \#^n \BC \BP^2,\quad W_{n}^{\bullet} \coloneq M^{\bullet} \#^n \BC \BP^2 \backslash M_n^{\bullet}.$$
Then the link $L_{C_n}^{\bullet} : M^{\bullet}_n \sqcup -M' \hookrightarrow \#^n \BC \BP^2$ is determined by the decomposition
\begin{equation}\label{5.1}
\#^n \BC \BP^2 = M^{\bullet}_n \cup W_{n}^{\bullet} \cup W_{C_0} \cup -M'.
\end{equation}

\begin{proof}[Proof of Theorem~\ref{thm3}]
We wish to show that $X \coloneq \#^n \BC \BP^2 \backslash L_{C_n}^{\bullet}(M_n^{\bullet})$ is a Stein manifold. To that end, we describe its Kirby diagram. It follows from~\eqref{5.1} that $X$ can be written as
\begin{equation}\label{5.2}
X = -M' \cup W_{C_0} \cup W_n^{\bullet} = -M_0 \cup W_n^{\bullet},\end{equation}
where the second equality follows from the fact that $L_{C_0}$ satisfies $S^4 \backslash L_{C_0}(M_0) = -M' \cup W_{C_0} = -M_0$.

In order to obtain a Kirby diagram of $-M_0 \cup W_n^{\bullet}$, we look at where the belt spheres of the 2-handles of $W_n^{\bullet}$ are sent under the gluing map $\phi : \partial_+ W_n^{\bullet} \rightarrow \partial M_0$, which decomposes as follows:
$$\phi : \partial_+W_n^{\bullet} \xrightarrow{=} \partial (M_n^{\bullet} \cup W_n^{\bullet}) \xrightarrow{\phi_n} \partial (M^{\bullet} \#^n \BC \BP^2) \xrightarrow{=} \partial M^{\bullet} \xrightarrow{\theta^{-1}} \partial M_0.$$
Here we write an equality above the first arrow because we are working with the relative Kirby diagram of $(W_n^{\bullet}, \partial M_n^{\bullet})$ coming from the diagram of $M_n^{\bullet}$.
It becomes clear that we only need to analyze the diffeomorphism $\phi_n$, which consists of $|n|$ Rolfsen twists along the $(-1)$-framed, negatively oriented meridians $-\mu^-_n$ of $K_n^{\bullet}$, where $K_n^{\bullet}$ is the $n$-framed knot determining $M_n^{\bullet}$. These meridians are negatively oriented because we slid the 2-handle of $M^{\bullet}$ over the $(-1)$-framed 2-handles, resulting in negative linking. 

Through careful analysis of the slam-dunk move, one verifies that $\phi_n$ sends the oriented, 0-framed meridians of $-\mu^-_n$ to the oriented, $(-1)$-framed meridians $\mu^-$ of the knot $K^{\bullet}$ defining $M^{\bullet}$. Applying the zero-dot exchange $\theta^{-1}$, these are sent to the oriented, $(-1)$-framed meridians $\mu^-_{\bullet}$ of the dotted unknot in the surgery diagram of $\partial M_0$. Lastly, after reversing the orientation of $M_0$, the $\mu^-_{\bullet}$ are sent to the negatively oriented, $(-1)$-framed meridians $-\mu^-_{\bullet}$ of the dotted unknot in the surgery diagram of $- \partial M_0$. It follows that the Kirby diagram of $X = -M_0 \cup W_n^{\bullet}$ consists of the diagram of $-M_0$ together with $|n|$ negatively oriented, $(-1)$-framed meridians of the dotted unknot, as in Figure~\ref{fig7} (left).

\begin{figure}[htbp]
  \centering
  \includegraphics[width=0.25\textwidth]{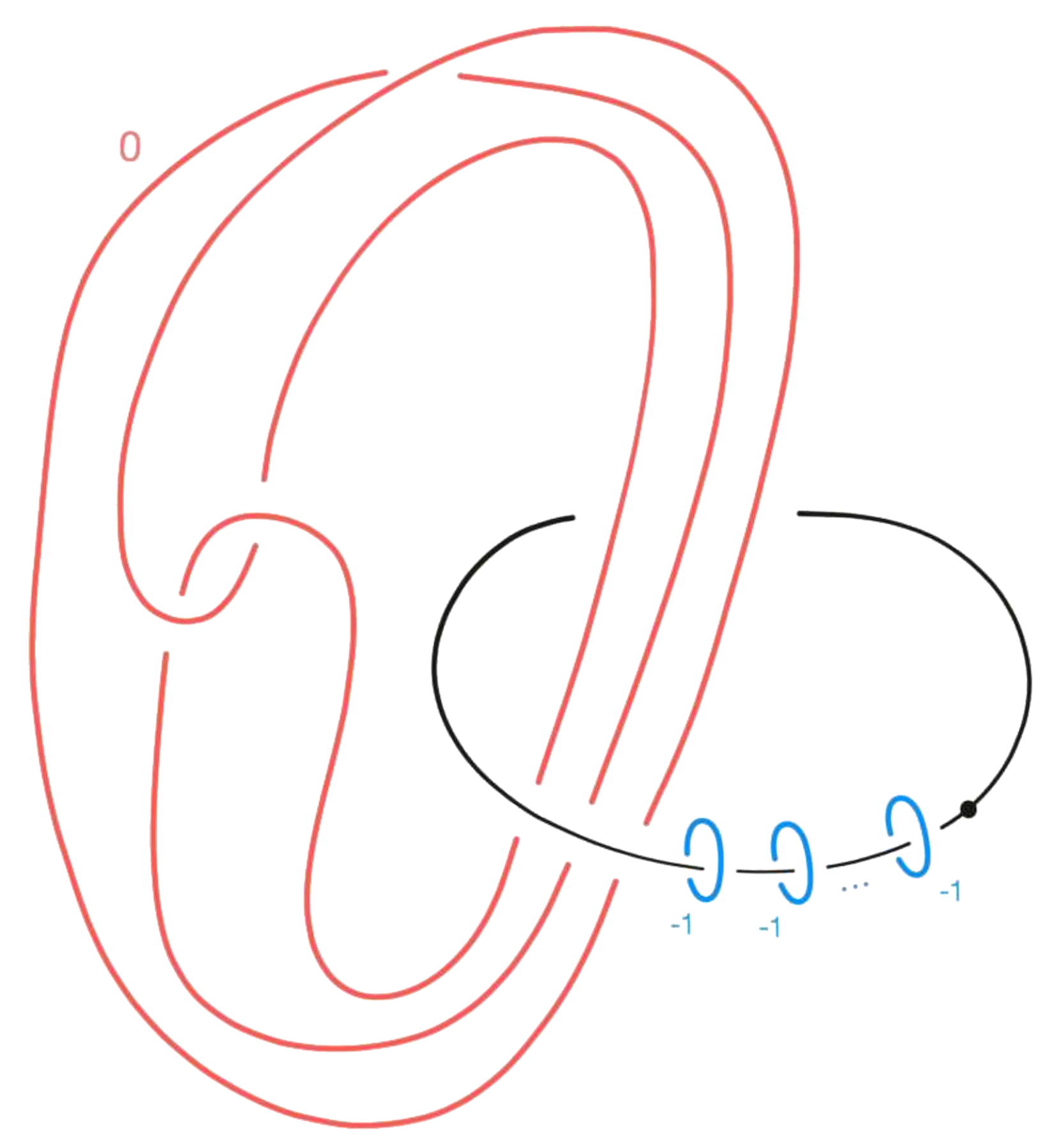} 
  \hspace{2.5 cm}
  \includegraphics[width=0.25\textwidth]{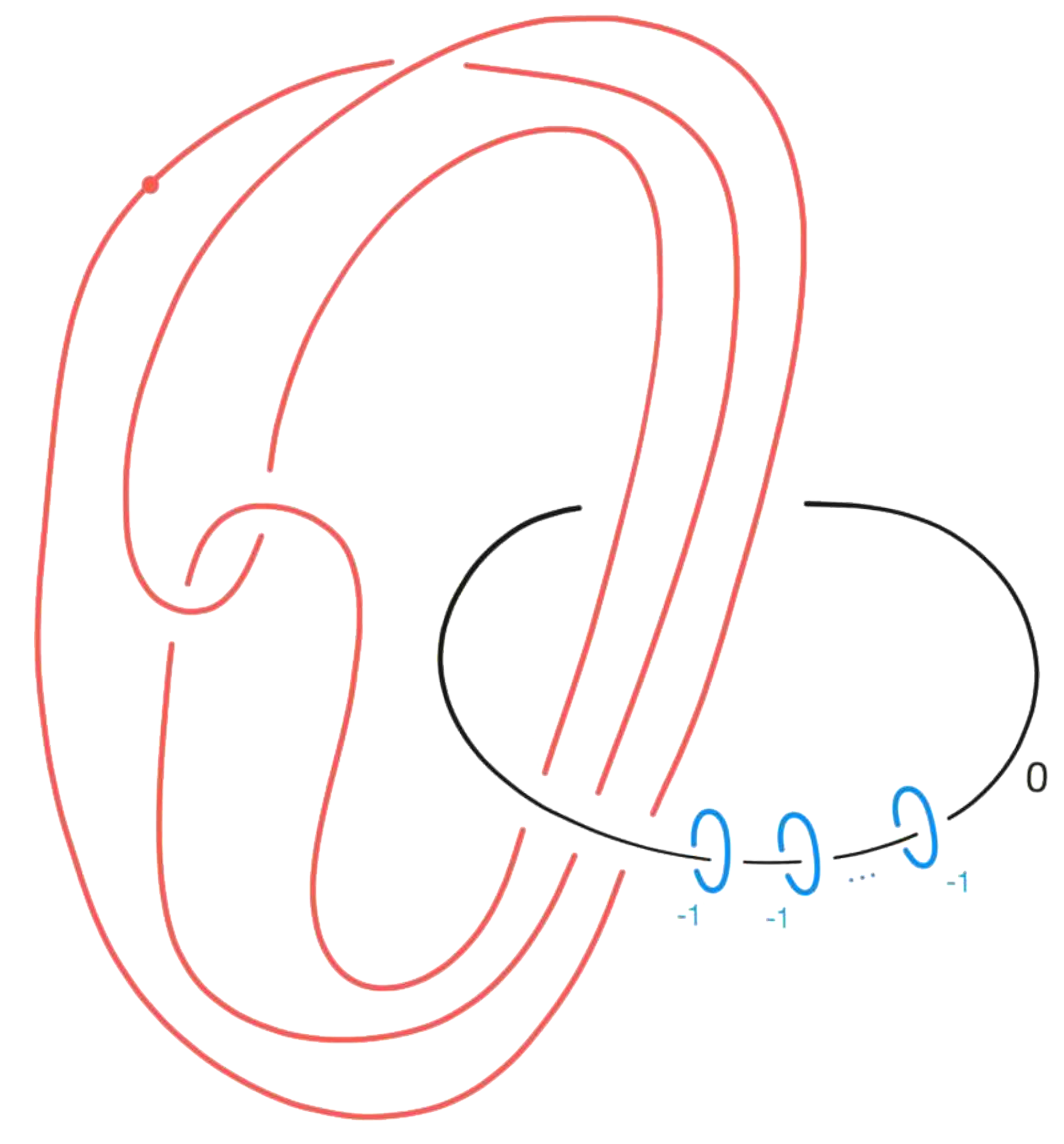}
  \caption{Left: a Kirby diagram of the complement $\#^n \BC \BP^2 \backslash L_{C_n}^{\bullet}(M_n^{\bullet})$.\\ Right: a Kirby diagram of the complement of the standard embedding $M_n^{\bullet} \hookrightarrow \#^n \BC \BP^2$.}
  \label{fig7}
\end{figure}

We claim that this Kirby diagram satisfies the Eliashberg criterion and therefore represents a Stein manifold. Indeed, passing to the 3-ball notation for the 1-handle, we have $\on{tb}(-K_0) > 0 = \on{fr}(-K_0)$ by assumption, where $\on{fr}$ denotes the framing. Furthermore, $$\on{tb}(-\mu_{\bullet}^-) = \on{wr}(-\mu_{\bullet}^-) - \on{c}(-\mu_{\bullet}^-)/2 = 0 > -1 = \on{fr}(-\mu_{\bullet}^-),$$
where $\on{wr}$ and $\on{c}$ stand for the writhe and cusp number, respectively. 

On the other hand, the \textit{standard} embedding $M^{\bullet}_n \hookrightarrow \#^n \BC \BP^2$ must have complement $X' = \#^n \BC \BP^2 \backslash M^{\bullet}_n = -M^{\bullet}_{n} \#^n \BC \BP^2$. Since $X'$ contains a smoothly embedded sphere of self-intersection $-1$,  it cannot admit a Stein structure (see \cite[Theorem~11.4.2]{Gompf-Stipsicz}). Sliding the 2-handle of $-M_n^{\bullet}$ once over each of the $(-1)$-framed 2-handles, one obtains a Kirby diagram for $X'$ as in Figure~\ref{fig7} (right). Since $\partial X = -\partial M_n^{\bullet} = \partial X'$, while $X$ and $X'$ are simply connected and have isomorphic intersection forms, we can again apply Boyer's results to conclude that $X$ and $X'$ are homeomorphic. It follows that $(X, X')$ is an exotic pair. Moreover, it is clear from Figure~\ref{fig7} that $X$ contains an embedded copy of $-M_0$, and replacing it by $-M^{\bullet}$ via $-\theta$ results in $X'$.

Lastly, we show that the pair $(X, X')$ is also related by a replacement $-M' \rightsquigarrow -N'$, where $N' \coloneq S^4 \backslash L^{\bullet}_{C_0}(-M') = M^{\bullet} \cup W_{C_0}$, as in Theorem~\ref{thm4}. Replacing $M'$ by $N'$ in~\eqref{5.2}, we obtain
$$-N' \cup W_{C_0} \cup W_n^{\bullet} = -M^{\bullet} \cup_{\phi_n} W_n^{\bullet} = -M^{\bullet}_n \#^n \BC \BP^2 = X',$$
where the first equality will follow from Lemma~\ref{lemD}, while the second equality follows from our investigation of the map $\phi_n$ and the Kirby diagram of $X'$.

We have seen that removing $-M'$ from $X$ and gluing in $-N'$ produces $X'$. Since the pair $(X, X')$ is exotic, we conclude that $-M' \rightsquigarrow -N'$ is an effective replacement in the language of \cite{Equivariant}. In particular, the boundary map $\theta' : \partial M' \rightarrow \partial N'$ does not extend to a diffeomorphism $\Theta' : M' \xrightarrow{\cong} N'$. 
\end{proof}

Lastly, we prove Corollary~\ref{cor3}.

\begin{proof}
First, assume $n < 0$. Let $f : M_n^{\bullet} \hookrightarrow \#^n \BC \BP^2$ be the standard embedding, and let $g : M_n^{\bullet} \hookrightarrow \#^n \BC \BP^2$ be the restriction of $L_{C_n}^{\bullet}$ to $M_n^{\bullet}$. If there were a diffeomorphism $\Phi : \#^n \BC \BP^2 \rightarrow \#^n \BC \BP^2$ sending $f(\partial M_n^{\bullet})$ to $g(\partial M_n^{\bullet})$, it would restrict to a diffeomorphism between $\#^n \BC \BP^2 \backslash L_{C_n}^{\bullet}(M_n^{\bullet})$ and the complement of the image of $f$, contradicting Theorem~\ref{thm3}. 

For the positive case, we simply reverse the orientation of $\#^n \BC \BP^2$. While $\#^{-n} \BC \BP^2 \backslash L_{C_n}^{\bullet}(M_n^{\bullet})$ is no longer Stein, it is still an exotic copy of $\#^{-n} \BC \BP^2 \backslash f(M_n^{\bullet})$ since their negative copies are distinguished.
\end{proof}

\section{Exotica and ribbon homology cobordisms}\label{6}
Before proving Theorem~\ref{thm4}, we recall the setting.
As usual, we assume that $(M, M')$ is a ribbon pair of Mazur manifolds, related by a ribbon concordance $C \subset S^1 \times S^2 \times I$.

Let $\varphi : M \hookrightarrow S^4$ be a smooth embedding such that the complement $N \coloneq -S^4 \backslash \varphi(M)$ lies in $\A$. Since $N$ must be simply connected, it follows from Proposition~\ref{topknot} that $\varphi$ is topologically isotopic to the standard embedding $\varphi_M : M \hookrightarrow S^4$. 

We assume that $\varphi$ is not smoothly isotopic to $\varphi_M$. By Proposition~\ref{propEx}, this is equivalent to requiring that the induced boundary map $\theta : \partial M \rightarrow \partial N$ does not extend to a diffeomorphism $\Theta : M \xrightarrow{\cong} N$, although it must extend to a homeomorphism $M \approx N$.

Letting $L_C$ be the link provided by Theorem~\ref{thm1}, we consider a link $$L_{C, \theta} : N \sqcup -M' \hookrightarrow S^4$$
obtained by removing $M$ and gluing in $N$ using the map $\theta$. The ambient space remains $S^4$ because $S^4 \backslash L_C(M) = -M$ and $N \cup_{\theta} -M = S^4$.

Next, denote $\varphi' \coloneq L_{C, \theta}|_{M'}$ and set $N' \coloneq S^4 \backslash \varphi'(-M') = N \cup_{\theta} W_C$. 
Letting $\theta' : \partial M' \rightarrow \partial N'$ be the boundary diffeomorphism associated to the embedding $\varphi' : M' \hookrightarrow S^4$, we wish to show that $\theta'$ does not extend to a diffeomorphism $\Theta' : M' \xrightarrow{\cong} N'$, although again it must extend to a homeomorphism $M' \approx N'$. With that in mind, we apply the replacement $M' \rightsquigarrow_{\theta'} N'$ to the link $L_{C, \theta}$, obtaining a link
$$L_{C, \theta}' : N \sqcup -N' \hookrightarrow S^4,$$
where the ambient space is $S^4$ because it is the double of $N' = N \cup W_C$, and $N' \in \A$ since $N \in \A$.

\begin{proof}[Proof of Theorem~\ref{thm4}]
Assume, to the contrary, that $\theta'$ extends to a diffeomorphism $\Theta' : M' \xrightarrow{\cong} N'$. Then we can extend $\Theta'$ to a diffeomorphism
$$\Theta : M' \cup - W_C \xrightarrow{\cong} N' \cup -W_C = N \cup_{\theta} W_C \cup -W_C.$$
As the name suggests, $\Theta$ is, in fact, a diffeomorphism $M \cong N$ extending $\theta$, providing the necessary contradiction.

Since we have seen that $M' \cup -W_C = M$, it would suffice to show that
$$N \cup W_C \cup -W_C = N,$$
thus the theorem follows from the following lemma.
\end{proof}

\begin{lem}\label{lemD}

There exists a diffeomorphism $N \cup W_C \cup -W_C \cong N$ restricting to $\theta$ on the boundary, where we identify $\partial_+(-W_C) = -(\partial_-W_C) = \partial M$.
\end{lem}
\begin{proof}
For simplicity, we first consider the case when $M \in \M_0$ and $N$ is obtained from $M$ by a zero-dot exchange. In our pictures, we will again assume that $M$ is the Akbulut cork, but the argument is the same for every $M \in \M_0$.\\

\textbf{Special case:} As we saw in the proof of Theorem~\ref{thm1}, a Kirby diagram of $M \cup W_C$ appears as in Figure~\ref{fig3} (top right). Replacing $M$ by $N$ using the zero-dot exchange, we obtain a Kirby diagram for $N \cup W_C$ which appears as in Figure~\ref{fig8} (left), where we denote the 1- and 2-handles of $N$ by $y'$ and $x'$ because they are obtained from the 1- and 2-handles $x$ and $y$ of $M$ by a zero-dot exchange. Since the boundary diffeomorphism $\partial_+W_C \rightarrow \partial (N \cup W_C)$ is canonical, a Kirby diagram of $N \cup W_C \cup -W_C$ is obtained by adding 0-framed 2-handles $z_i$ along the belt spheres of the 2-handles $y_i$ of $W_C$, hence it looks like Figure~\ref{fig8} (right). Notice that after the zero-dot exchange the $y_i$, which used to link the 2-handle $y$ of $M$, now link the 1-handle $y'$ of $N$ at the meridians $\mu'$. Our goal is to unlink the $y_i$ from $y'$ while preserving the rest of the picture (see Figure~\ref{fig9}).

First, we slide $y_1$ over $x'$ to obtain $\on{lk}(y_1, y') = 0$, as in Figure~\ref{fig10}. Letting $s_1$ be the strand of $y_1$ which follows $x'$, we wish to move it outside of $y'$. We claim that this can be accomplished if we allow self-crossing changes for $s_1$; see Figure~\ref{fig11} for an example. Each crossing change can be realized by sliding $y_1$ over $z_1$ using an appropriate arc. Since that changes the framing of $y_1$, we also slide it under $z_1$ using a trivial arc to correct the framing. This process is described in \cite[Proposition~3]{Kirby}.

Repeating this for the other indices, we are able to unlink each $y_i$ from $y'$ at the meridian $\mu'$. After that, the $y_i$ can be isotoped across the paths $\gamma_i$, resulting in the Kirby diagram of Figure~\ref{fig12}. It remains to cancel the handle pairs $x_i, y_i$, and the handles $z_i$ with the 3-handles coming from $-W_C$, leaving us with a Kirby diagram of $N$. One can easily check that this diffeomorphism extends $\theta$,
which completes the proof for the case when $M$ and $N$ are related by a zero-dot exchange. \\

\textbf{General case:} While the Kirby diagram of $M \cup W_C$ still looks like Figure~\ref{fig3} (top right) --- though the 2-handle $y$ may not be unknotted --- the replacement $M \rightsquigarrow N$ is now more complicated. This is because the boundary map $\theta : \partial M \rightarrow \partial N$ is abstract, and so even though the relative handle structure $(W_C, \partial N)$ is the same as $(W_C, \partial M)$, we cannot track the 2-handles of $W_C$ in $N \cup W_C$ explicitly. However, the idea of the special case works regardless.

We can envision a Kirby diagram $K_1$ of $N \cup W_C$ as follows: start with a Kirby diagram $K_N$ of $N$ and add $m$ 1-handles $x_i'$. Next, we add $m$ 2-handles $y_i'$ whose attaching circles are described as follows. Recall that in the Kirby diagram of Figure~\ref{fig3} (top right) the attaching circle of $y_i$ is obtained from the meridians $\mu_K$ and $\mu_i$ (of $y$ and $x_i$, respectively) by connecting them using the band $\gamma_i$. Let $\mu'$ and $\gamma_i'$ be the images of $\mu_K$ and $\gamma_i$ in $K_N$, obtained by applying the map $\theta : \partial M \rightarrow \partial N$. For convenience, let $\mu'_i$ denote the copy of $\mu'$ near the endpoint of $\gamma_i'$. Since $\theta$ can be assumed to be supported away from the 1-handles $x_i$, we may preserve the notation for the $\mu_i$. Now, it can be seen that the attaching circles of the 2-handles $y_i'$ are obtained by connecting $\mu'_i$ and $\mu_i$ using the bands $\gamma_i'$. This completes our description of the Kirby diagram $K_1$. 

The Kirby diagram $K_2$ of $N \cup W_C \cup -W_C$ is obtained from $K_1$ by attaching 0-framed 2-handles $z_i$ along the meridians of the $y_i'$, as before. Once again, our goal is to use these meridians to unlink the $y_i'$ from the handles of $K_N$. The difference is that now $\mu'_i$ and $\gamma_i'$ are more complicated, but it turns out that this is of no consequence. 

First, we observe that even though $\gamma_i'$ can be more complicated than $\gamma_i$ since we have applied the map $\theta$, it is still a ribbon band. In particular, if we could unlink the $\mu'_i$ portion of the $y_i'$ from the handles of $K_N$, we would be able to isotope the 2-handles $y_i'$ across $\gamma_i'$ until they appear as meridians of the 1-handles $x_i'$, just like in the special case. Then we would be able to cancel the extra handles, leaving the Kirby diagram $K_N$ of $N$. 

Let $a_1, \ldots, a_p$ and $b_1, \ldots, b_p$ be the 1- and 2-handles of $N$, respectively, where each $b_j$ must homotopically cancel $a_j$ since $N$ is AC. Note that the curves $\mu'_i$ \enquote{spell out} a word in the $a_j$. We can slide $y_i'$ over the $b_j$ where necessary until this word becomes algebraically trivial, such as $a_1 a_1^{-1}a_2 a_2^{-1}$. This means that the $\mu_i'$ go through the $a_j$ geometrically, but not algebraically. In particular, they can be isotoped away from the $a_j$ if we allow self-crossing changes. These can be realized by sliding $y_i'$ over and under $z_i$ as in the special case. It remains to unlink the $\mu_i'$ from the 2-handles $b_j$, which is accomplished by sliding $b_j$ over/under $z_i$. This completes the proof of the lemma.
\end{proof}

\begin{figure}[htbp]
  \centering
  \includegraphics[width=0.42\textwidth]{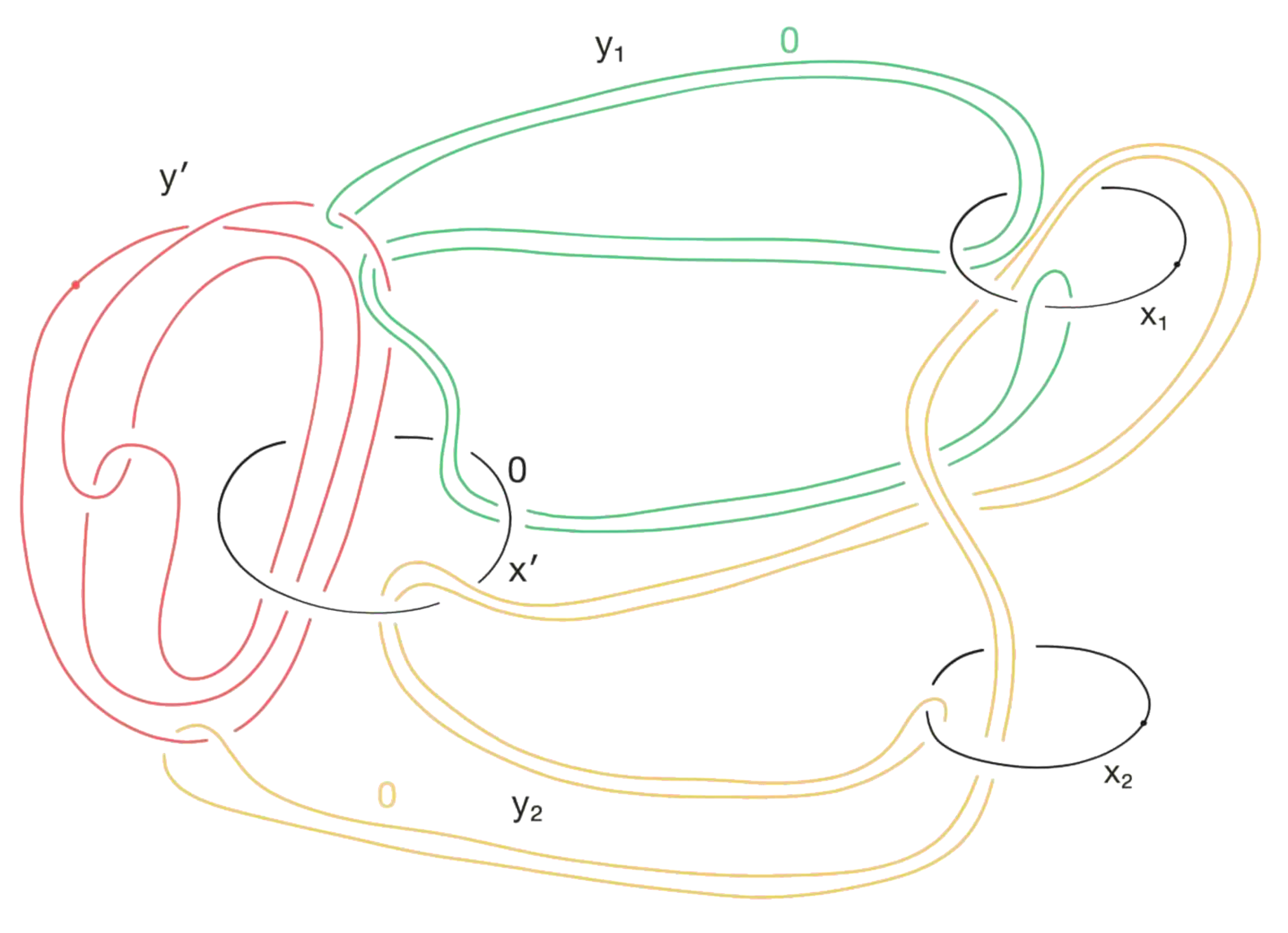} 
  \hspace{0.5 cm}
  \includegraphics[width=0.42\textwidth]{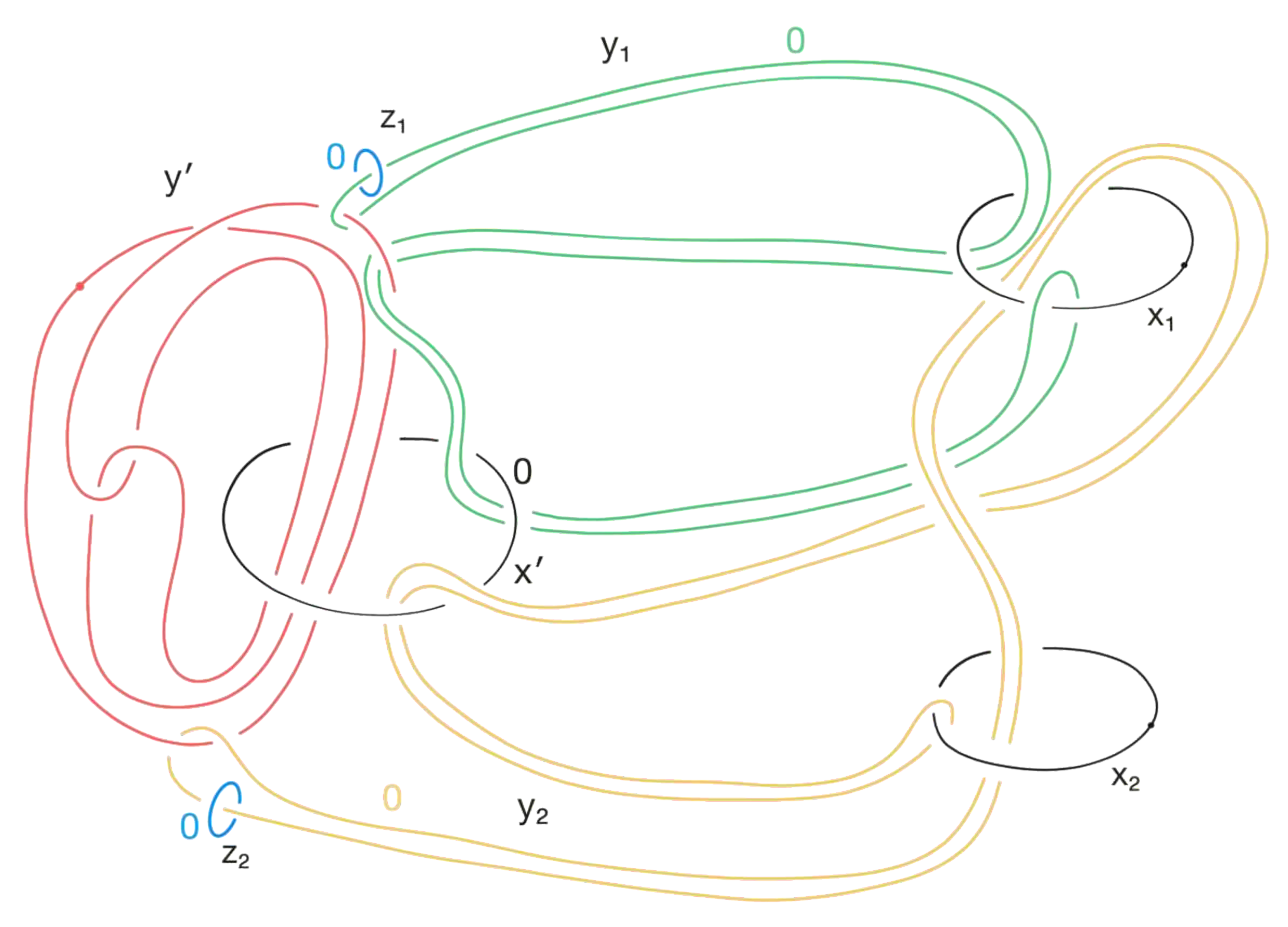}
  \caption{\centering Left: a Kirby diagram of $N \cup W_C$ in the special case.\\ Right: a Kirby diagram of $N \cup W_C \cup -W_C$ in the special case.}
  \label{fig8}
\end{figure}

\begin{figure}[htbp]
  \centering
  \includegraphics[width=0.4\textwidth]{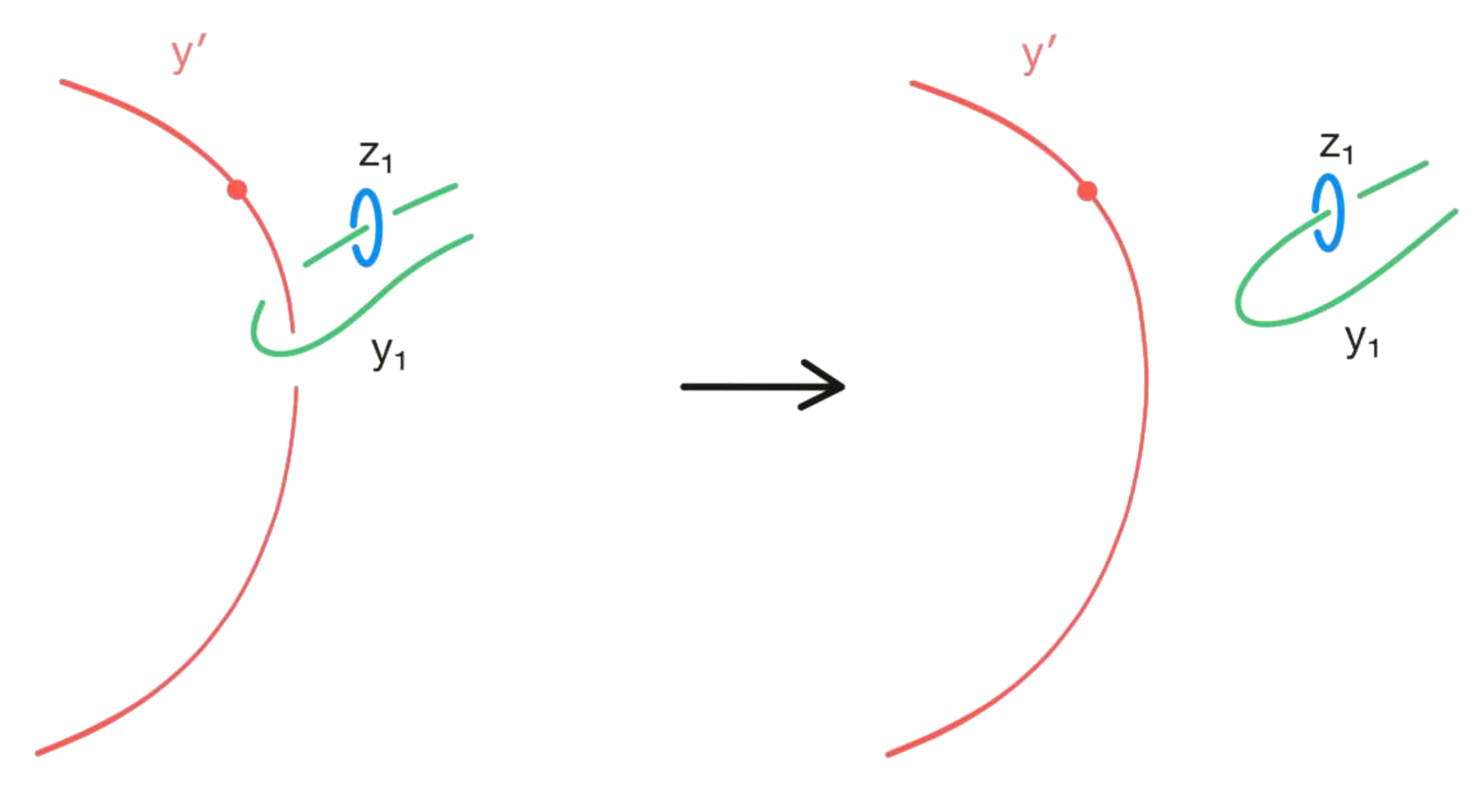} 
  \caption{The transformation used to prove the special case of Lemma~\ref{lemD}.}
  \label{fig9} 
\end{figure}

\begin{figure}[htbp]
  \centering
  \includegraphics[width=0.5\textwidth]{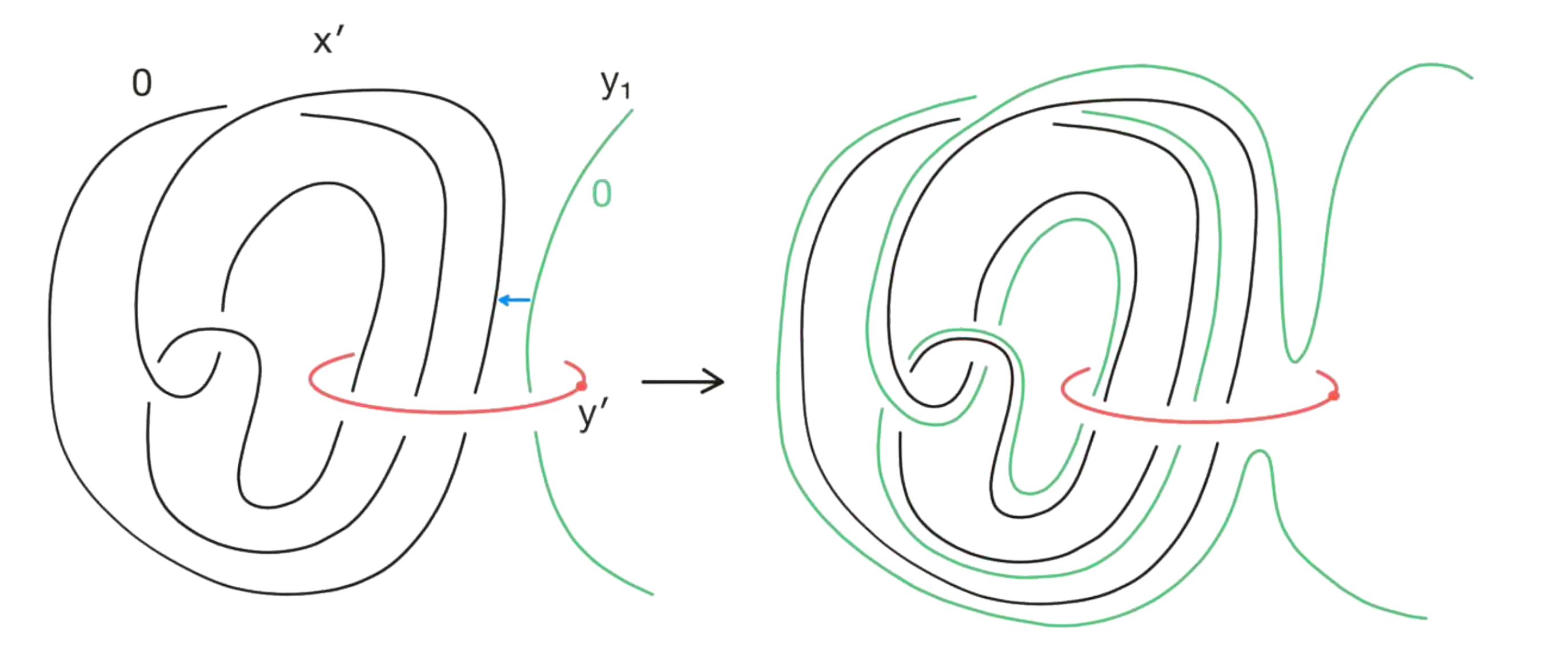}
  \caption{An example of a slide of $y_1$ over $x'$ that results in $\on{lk}(y_1, y') = 0$.}
  \label{fig10}
\end{figure}

\begin{figure}[htbp]
  \centering
  \includegraphics[width=0.4\textwidth]{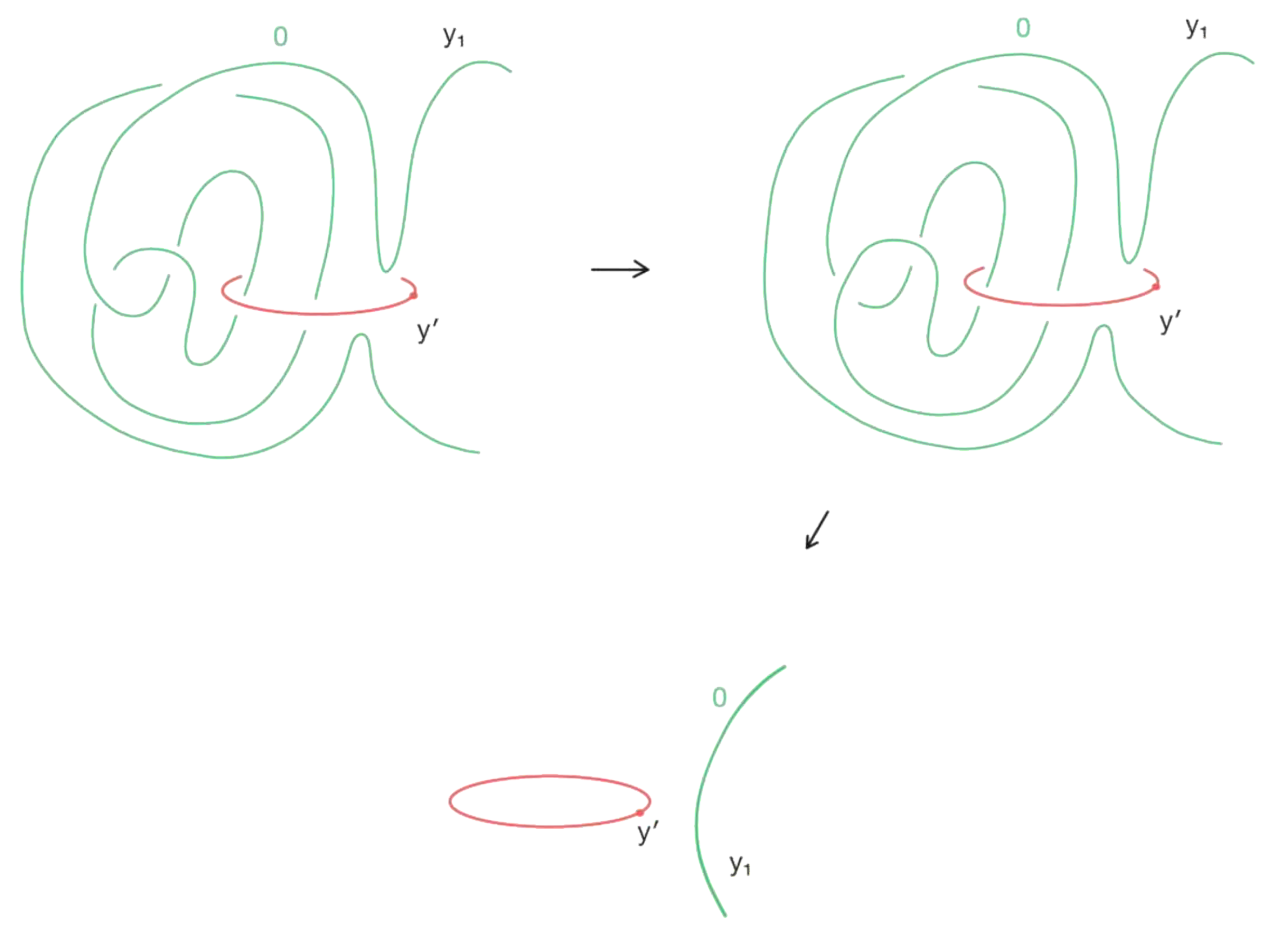}
  \caption{An example of how to unlink the strand $s_1$ from $y'$ using self-crossing changes.}
  \label{fig11}
\end{figure}

\begin{figure}[htbp]
  \centering
  \includegraphics[width=0.4\textwidth]{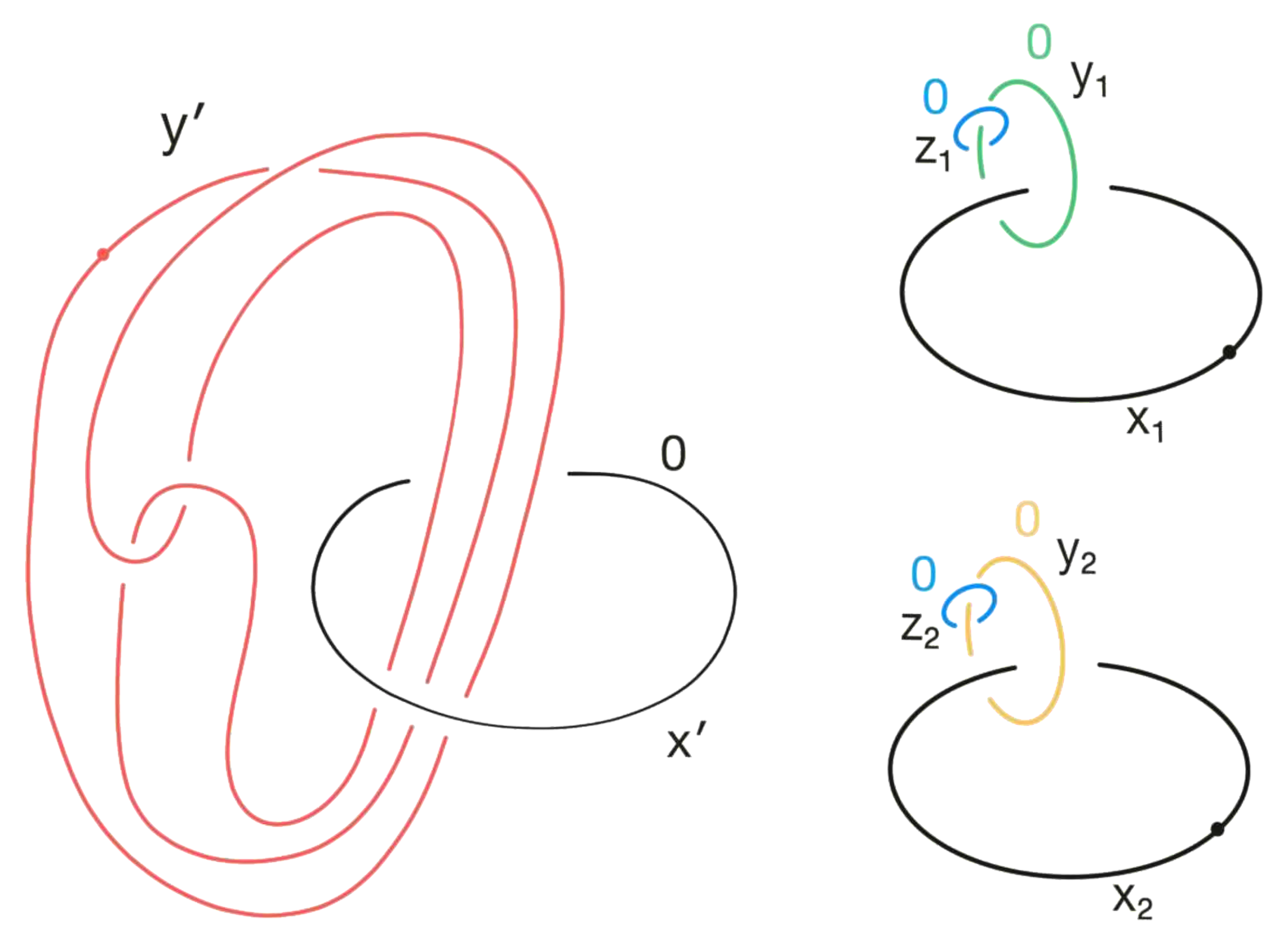}
  \caption{A Kirby diagram showing that $N \cup W_C \cup -W_C \cong N$ in the special case.}
  \label{fig12}
\end{figure}

\begin{remark}\label{rem6.2}
Since the embedding $\varphi' : M' \hookrightarrow S^4$ is exotic, the restriction $\varphi'|_{\partial M'}$ must be topologically but not smoothly isotopic to the restriction $\varphi_{M'}|_{\partial M'}$ of the standard embedding of $M'$ (compare to Remark~\ref{rem2.6}). Thus, Theorem~\ref{thm4} also provides interesting embeddings of the homology spheres $\partial M'$.
\end{remark}
\newpage

\section{Further questions}\label{7}
We note that Definition~\ref{ribbon} may seem somewhat artificial since it introduces a relation between smooth 4-manifolds which depends on their handle decomposition. Instead, one could relate two contractible 4-manifolds $X, X' \in \C$ by a general ribbon homology cobordism from $\partial X$ to $\partial X'$. While this approach may be more natural for general contractible 4-manifolds, Mazur manifolds are a different story. Their name implicitly suggests that they have handle decompositions with a single 2-handle. Hence, it makes sense to explicitly choose such decompositions and consider a relation based on them instead.

Nonetheless, one could still connect Mazur manifolds $M$ and $M'$ by a general ribbon homology cobordism $W_R$ from $\partial M$ to $\partial M'$, which leads to a link
$$L_R : M \sqcup -M' \hookrightarrow \Sigma = M \cup W_R \cup -M'.$$
However, we are unable to show that the homotopy 4-sphere $\Sigma$ is standard in general. On the other hand, the smooth obstruction still applies since $W_R$ is ribbon; thus $L_R$ cannot be smoothly split by an $S^3$. In fact, the topological obstruction also works since even a general ribbon homology cobordism cannot kill the fundamental group of $\partial M$ by \cite{Metatheorem}. However, we do not highlight this observation because we are not aware of any method of finding such $W_R$, other than surgering a ribbon concordance $C$ between the knots $K$ and $K'$ defining the Mazur manifolds as in Theorem~\ref{thm1}.

Another issue is that Theorem~\ref{thm4} breaks down in this case. As we noted in the introduction, this result essentially says that exotica travels across the ribbon homology cobordisms $W_C$. But these are obtained by surgering the ribbon concordances $C$, and so are very particular. It is true that for a general ribbon homology cobordism $W_R$ one can still consider the different links used in the proof of Theorem~\ref{thm4}, while Lemma~\ref{lemD} still applies as long as $W_C$ is \enquote{AC relative to $\partial M$}. However, we do not have a diffeomorphism $W_R \cup -M' \cong -M$. Similarly, we would not have $W_R \cup -N' \cong -N$, and so it does not follow that the replacement $M' \rightsquigarrow_{\theta'} N'$ does anything interesting smoothly.\\

Alternatively, one could ask whether the ribbon condition can be relaxed. It is shown in \cite{Yildiz} that any two freely homotopic knots $K$ and $K'$ in $S^1 \times S^2$ are concordant. It follows that given any pair of Mazur manifolds $M = M(K)$ and $M' = M(K')$ with the same framing, one can find a smooth concordance $C \subset S^1 \times S^2 \times I$ from $K$ to $K'$. Surgering $C$ as before, we obtain a homology cobordism $W_C$ from $\partial M$ to $\partial M'$, although it need not be ribbon. All the same, one obtains a link $$L_C : M \sqcup -M' \hookrightarrow \Sigma =  M \cup W_C \cup -M'.$$ 

The first caveat is that without the ribbon assumption we are unable to show that $\Sigma$ is diffeomorphic to $S^4$. Of course, depending on perspective, non-split links in a potentially exotic homotopy 4-sphere can be even more interesting, but there is a second caveat. If the concordance $C$, and hence the homology cobordism $W_C$, is not ribbon, then we have no control over the fundamental group of $W_C$ or the induced cobordism maps $F_{W_C}^{\circ}$, which prevents us from obstructing splittings of the link.\\

We run into a somewhat similar problem if we try to relax the condition $n \neq 0$ of Theorem~\ref{thm2}.
The main idea of the construction is that once we \enquote{lift} the links $L_C$ to $L_{C, n} \subset \#^n \BC \BP^2$, we are able to use the second homology to kill the fundamental group of the complement $\#^n \BC \BP^2 \backslash L_{C, n}$. One can ask whether the second homology is necessary here. We claim that the answer is essentially yes.

Assume we have a smooth link $L : M \sqcup -M' \hookrightarrow S^4$ such that the complement $W = S^4 \backslash L$ is a ribbon homology cobordism. Then \cite[Theorem~1.14]{Metatheorem} implies that the inclusion map $\pi_1(\partial M) \rightarrow \pi_1(W_C)$ is injective, just like in the proof of Theorem~\ref{thm1}. In particular, $W_C$ is not simply connected, thus $L$ cannot be topologically isotopic to the unlink.


Of course, this does not mean that there are no exotic (2-component) links of Mazur manifolds in $S^4$. However, our obstruction relies heavily on the ribbonness of the cobordism, which, in particular, allows us to avoid any explicit computation of Heegaard Floer homology. On the other hand, in order to obstruct a general homology cobordism from admitting a splitting by $S^3$ using the cobordism map (on $HF$, or perhaps another invariant), one would likely need a more specific construction and computation. In any case, we pose the following question.

\begin{quest}\label{Q1} Find a smooth link $L : M \sqcup M' \hookrightarrow S^4$ of Mazur manifolds such that $L$ is split by a topological $S^3$, but not by a smooth one. 
\end{quest}




Another natural question is about exotic links in indefinite 4-manifolds, the simplest being $S^2 \times S^2$. One can lift the links $L_C \subset S^4$ from Theorem~\ref{thm1} to topologically standard links $L_{C}' \subset S^2 \times S^2$ in a similar way: start with the diagram of $M \cup W_C \subset S^4$ as in Figure~\ref{fig3} (top right), attach a $0$-framed 2-handle $z_1$ along a meridian of $y$ and then add another $0$-framed 2-handle $z_2$ at the meridian of $z_1$. Since a Dehn surgery on $z_1$ and $z_2$ preserves the 3-manifold by a slam-dunk move, we obtain a link $L_{C}' : M \sqcup -M' \hookrightarrow S^2 \times S^2$. Note that unlike Theorem~\ref{thm2}, this time the framing of $M$ does not change. Since $z_1$ kills $\pi_1(\partial M)$, the complement $W_C' = S^2 \times S^2 \backslash L_C'$ is simply connected. It follows from Proposition~\ref{biglink} that $L_C'$ is topologically isotopic to the unlink. 

The main difference with the case of $\#^n \BC \BP^2$ is that our smooth obstruction fails. Decomposing $W_C'$ as $W' \cup W_C$ like before, we obtain a smooth cobordism $W'$ from $\partial M$ to $\partial M$ consisting of two $0$-framed 2-handles. However, the induced map $F_{W'} : HF^+_{\mathrm{red}}(\partial M) \rightarrow HF^+_{\mathrm{red}}(\partial M)$ happens to vanish. This is essentially because $S^2 \times S^2$, and thus also $W'$, contains a smoothly embedded, homologically essential sphere with self-intersection $0$. We ask the following question.

\begin{quest}\label{Q3} 
Find a smooth link $L : M \sqcup -M' \hookrightarrow S^2 \times S^2$ that is split by a topological $S^3$, but not by a smooth one. 
\end{quest}

We draw the reader's attention to the fact that we have implicitly assumed that all Mazur manifolds are oriented and their embeddings are orientation-preserving. This is why we write $M \sqcup -M$ and not $M \sqcup M$, for example. Similarly, the link $L_C : M \sqcup -M' \hookrightarrow S^4$ from Theorem~\ref{thm1} is orientation-preserving, and if we viewed it as a link of $M \sqcup M'$, its restriction to $M'$ would be orientation-reversing. With our current methods, we are unable to account for this orientation change. Letting $(M, M')$ be a ribbon pair via a concordance $C$, one can ask:

\begin{quest}\label{quest:orientation2}
Does there exist an orientation-preserving link $L : M \sqcup M' \hookrightarrow S^4$ that is not split by an $S^3$? What about a link $L_n : M_n \sqcup M' \hookrightarrow \#^n \BC \BP^2$?
\end{quest}

Note that if $C$ is the trivial concordance, then $M' \cong M$; thus the first question asks, in particular, for a non-split, orientation-preserving link $L : M \sqcup M \hookrightarrow S^4$.\\

A comparison with the Golla--Marengon links is due. Similarly to them, we use the fundamental group to obstruct topological splittings by $S^3$. In certain cases, they obstruct smooth splittings by some homology spheres using Heegaard Floer homology, which has been an inspiration for our ideas. However, they do not obtain links that are split topologically but not smoothly, which we achieve in $\#^n \BC \BP^2$ for $n \neq 0$. On the other hand, while they only state their results for 2-component links, the construction generalizes to links with any number of components in a straightforward way. In contrast, our methods are inherently binary and only work for 2-component links. The following question was suggested to the author by Peter Kronheimer.

\begin{quest}\label{Q4}
    For $m > 2$, find a link of Mazur manifolds $L : M_1 \sqcup M_2 \sqcup \cdots \sqcup M_m \hookrightarrow S^4$ such that for every $i \neq j$ the images $L(M_i)$ and $L(M_j)$ are not split (topologically or smoothly) by an $S^3$.
\end{quest}

A natural starting point would be to assume $m = 3$ and that the $M_i$ are diffeomorphic up to orientation reversal. Then two of them must have the same orientation. Thus, $L$ would have to restrict to a non-split link $M \sqcup M \hookrightarrow S^4$, whose existence remains a mystery to us.

\bibliographystyle{amsalpha}
\bibliography{bib}

\bigskip
\bigskip

\noindent \textsc{Department of Mathematics, Columbia University, New York, NY 10027, USA} \\
\textit{Email address}: \texttt{sn3018@columbia.edu}

\end{document}